\documentclass{amsart}
\usepackage{amsmath,amsfonts,amssymb,subfigure}
\usepackage{a4wide}
\usepackage{comment}

\usepackage{graphicx}
\usepackage{float}

\usepackage{enumerate}
\restylefloat{figure}

\newtheorem{lemma}{Lemma}[section]

\newtheorem{proposition}[lemma]{Proposition}
\newtheorem{theorem}[lemma]{Theorem}
\newtheorem{corollary}[lemma]{Corollary}
\theoremstyle{definition}
\newtheorem{definition}[lemma]{Definition}
\newtheorem{example}[lemma]{Example}

\theoremstyle{remark}
\newtheorem{remark}[lemma]{Remark}

\numberwithin{equation}{section} \numberwithin{table}{section}

\newcommand{\Mat}{\mathbf{M}_2(\mathbb{R})}
\newcommand{\Matd}{\mathbf{M}_d(\mathbb{R})}
\newcommand{\RR}{\mathbb{R}}
\newcommand{\vvv}{{|\!|\!|}}
\newcommand{\Lvvv}{\left|\!\left|\!\left|}
\newcommand{\Rvvv}{\right|\!\right|\!\right|}

\newcommand{\NN}{\mathbb{N}}

\newcommand{\supp}{\mathrm{supp}\,}

\newcommand{\tr}{\mathrm{tr}\,}

\begin{document}
\title[Counterexample to the finiteness conjecture]{An explicit counterexample to the Lagarias-Wang finiteness conjecture}

\dedicatory{To the memory of I.~M.~Gelfand}


\author{Kevin G. Hare}
\address{Department of Pure Mathematics, University of Waterloo, Waterloo, Ontario, Canada N2L 3G1.}
\email{kghare@uwaterloo.ca}
\thanks{Research of K. G. Hare supported, in part, by NSERC of Canada. Research of I. D. Morris supported by the EPSRC grant EP/E020801/1.}
\author{Ian D. Morris}
\address{Dipartimento di Matematica, Universit\`a di Roma Tor
Vergata, Via della Ricerca Scientifica, 00133 Roma, Italy.}
\email{morris@mat.uniroma2.it}
\author{Nikita Sidorov}
\address{School of Mathematics, University of Manchester, Oxford Road, Manchester M13 9PL, United Kingdom.}
\email{sidorov@manchester.ac.uk}
\author{Jacques Theys}
\address{BNP Paribas Fortis, 3, rue Montagne du Parc, B-1000 Bruxelles}
\email{jacques.theys@gmail.com}
\date{\today}

\begin{abstract}
The joint spectral radius of a finite set of real $d \times d$ matrices is defined to be the maximum possible exponential rate of growth of long products of matrices drawn from that set. A set of matrices is said to have the \emph{finiteness property} if there exists a periodic product which achieves this maximal rate of growth. J.~C.~Lagarias and Y.~Wang conjectured in 1995 that every finite set of real $d \times d$ matrices satisfies the finiteness property. However, T.~Bousch and J.~Mairesse proved in 2002 that counterexamples to the finiteness conjecture exist, showing in particular that there exists a family of pairs of $2 \times 2$ matrices which contains a counterexample. Similar results were subsequently given by V.~D.~Blondel, J.~Theys and A.~A.~Vladimirov and by V.~S.~Kozyakin, but no explicit counterexample to the finiteness conjecture has so far been given. The purpose of this paper is to resolve this issue by giving the first completely explicit description of a counterexample to the
Lagarias-Wang finiteness conjecture. Namely, for the set
\[
\mathsf{A}_{\alpha_*}:=
\left\{\left(\begin{array}{cc}1&1\\0&1\end{array}\right),
\alpha_*\left(\begin{array}{cc}1&0\\1&1\end{array}\right)\right\}
\]
we give an explicit value of
\[\alpha_* \simeq 0.749326546330367557943961948091344672091327370236064317358024\ldots
\]
such that $\mathsf{A}_{\alpha_*}$ does not satisfy the finiteness property.
\end{abstract}

\maketitle

\section{Introduction}

If $A$ is a $d\times d$ real or complex matrix and $\|\cdot\|$ is a matrix norm, the spectral radius $\rho(A)$ of the matrix $A$ admits the well-known characterisation
\[\rho(A) = \lim_{n \to \infty}\|A^n\|^{1/n},\]
a result known as Gelfand's formula. The \emph{joint spectral radius} generalises this concept to sets of matrices. Given a finite set of $d \times d$ real matrices $\mathsf{A} = \{A_1,\ldots,A_r\}$, we by analogy define the joint spectral radius $\varrho(\mathsf{A})$ to be the quantity
\[\varrho(\mathsf{A}):=\limsup_{n \to \infty} \max\left\{\left\|A_{i_1}\cdots A_{i_n}\right\|^{1/n} \colon i_j \in \{1,\ldots,r\}\right\},\]
a definition introduced by G.-C.~Rota and G.~Strang in 1960 \cite{RS} (reprinted in \cite{Rotacoll}). Note that the pairwise equivalence of norms on finite dimensional spaces implies that the quantity $\varrho(\mathsf{A})$ is independent of the choice of norm used in the definition.

 The joint spectral radius has been found to arise naturally in a range of mathematical contexts including control and stability \cite{Ba,DHX,Gu,Koz}, coding theory \cite{MO}, the regularity of wavelets and other fractal structures \cite{DL,DL0,Mae2,Prot}, numerical solutions to ordinary differential equations \cite{GZ}, and combinatorics \cite{BCJ,DST}. As such the problem of accurately estimating the joint spectral radius of a given finite set of matrices is a topic of ongoing research interest \cite{BN,GWZ,Koz4,Koz5,QBWF,Parr,BT1,Wirth}.


In this paper we study a property related to the computation of the joint spectral radius of a set of matrices, termed the \emph{finiteness property}. A set of $d \times d$ real matrices $\mathsf{A}:=\{A_1,\ldots,A_r\}$ is said to satisfy the finiteness property if there exist integers $i_1,\ldots,i_n$ such that $\varrho(\mathsf{A}) = \rho(A_{i_1}\cdots A_{i_n})^{1/n}$. The \emph{finiteness conjecture} of J.~Lagarias and Y.~Wang \cite{LW} asserted that every finite set of $d \times d$ real matrices has the finiteness property; a conjecture equivalent to this statement was independently posed by L.~Gurvits  in \cite{Gu}, where it was attributed to E.~S.~Pyatnitski\u{\i}.
In special cases, this finiteness property is known to be true, see for example  \cite{Cicone, Jungers}.
The existence of counterexamples to the finiteness conjecture was established in 2002 by T. Bousch and J. Mairesse \cite{BM}, with alternative constructions subsequently being given by V.~Blondel, J.~Theys and A.~Vladimirov \cite{BTV} and V.~S.~Kozyakin \cite{Koz3}. However, in all three of these proofs it is shown only that a certain family of pairs of $2 \times 2$ matrices must contain a counterexample, and no explicit counterexample has yet been constructed. The problem of constructing an explicit counterexample has been remarked upon as difficult, with G.~Strang commenting that an explicit counterexample may never be established \cite{Rotacoll}. In this paper, we resolve this issue by giving the first completely explicit construction of a counterexample to the Lagarias-Wang finiteness conjecture.

Let us define a pair of $2 \times 2$ real matrices by
\[A_0:=\left(\begin{array}{cc}1&1\\0&1\end{array}\right), \qquad A_1:=\left(\begin{array}{cc}1&0\\1&1\end{array}\right),\]
and for each $\alpha \in [0,1]$ let us define $\mathsf{A}_\alpha:=\{A_0,\alpha A_1\}$. The construction of Blondel-Theys-Vladimirov \cite{BTV} shows that there exists $\alpha \in [0,1]$ for which $\mathsf{A}_\alpha$ does not satisfy the finiteness property. The proof operates indirectly by demonstrating that the set of all parameter values $\alpha$ for which the finiteness property \emph{does} hold is insufficient to cover the interval $[0,1]$. In this paper we extend \cite{BTV} substantially by describing the behaviour of $\varrho(\mathsf{A}_\alpha)$ as the parameter $\alpha$ is varied in a rather deep manner. This allows us to prove the following theorem:
\begin{theorem}\label{counter}
Let $(\tau_n)_{n=0}^\infty$ denote the sequence of integers defined by $\tau_0:=1$, $\tau_1,\tau_2:=2$, and $\tau_{n+1}:=\tau_n\tau_{n-1}-\tau_{n-2}$ for all $n \geq 2$,\footnote{This is the sequence A022405 from Sloane's On-Line Encyclopedia of Integer Sequences.} and let $(F_n)_{n=0}^\infty$ denote the sequence of Fibonacci numbers, defined by $F_0:=0$, $F_1:=1$ and $F_{n+1}:=F_n+F_{n-1}$ for all $n \geq 1$. Define a real number $\alpha_* \in (0,1]$ by
\begin{equation}\label{eq:alphastar}
\alpha_*:=\lim_{n \to \infty}
\left(\frac{\tau_n^{F_{n+1}}}{\tau_{n+1}^{F_n}}\right)^{(-1)^n}= \prod_{n=1}^\infty \left(1-\frac{\tau_{n-1}}{\tau_n \tau_{n+1}}\right)^{(-1)^n F_{n+1}}.
\end{equation}
Then this infinite product converges unconditionally, and $\mathsf{A}_{\alpha_*}$ does not have the finiteness property.
\end{theorem}
The convergence in both of the limits given in Theorem~\ref{counter} is extremely rapid, being of order $O\left(\exp(-\delta \phi^n)\right)$ where $\delta>0$ is some constant and $\phi$ is the golden ratio. An explicit error bound is given subsequently to the proof of Theorem~\ref{counter}. Using this bound we may compute the approximation
\[\alpha_* \simeq 0.749326546330367557943961948091344672091327370236064317358024\ldots\]
which is rigorously accurate to all decimal places shown.

We shall now briefly describe the technical results which underlie the proof of Theorem~\ref{counter}. For each $\alpha \in [0,1]$ let us write $A_0^{(\alpha)}:=A_0$ and $A_1^{(\alpha)}:=\alpha A_1$ so that $\mathsf{A}_\alpha=\left\{A_0^{(\alpha)},A_1^{(\alpha)}\right\}$. The principal technical question which is addressed in this paper is the following: if we are given that for some finite sequence of values $u_1,\ldots,u_n \in \{0,1\}$, the matrix
\begin{equation}
\label{introproduct}A^{(\alpha)}_{u_n}A^{(\alpha)}_{u_{n-1}}\cdots A^{(\alpha)}_{u_2}A^{(\alpha)}_{u_1}
\end{equation}
is ``large'' in some suitable sense -- for example, if its spectral radius is close to the value $\varrho(\mathsf{A}_\alpha)^n$ - then what may we deduce about the combinatorial structure of the sequence of values $u_i$, and how does this answer change as the parameter $\alpha$ is varied? A key technical step in the proof of Theorem~\ref{counter}, therefore, is to show that the magnitude of the product \eqref{introproduct} is maximised when the sequence $u_1,u_2,\ldots,u_n$ is a \emph{balanced word}. This result depends in a rather essential manner on several otherwise unpublished results from the fourth named author's 2005 PhD thesis \cite{Theys}, which are substantially strengthened in the present paper.

In the following section we shall introduce the combinatorial ideas needed to describe balanced words. We are then able to state our main technical theorem, describe its relationship to previous research in ergodic theory and the theory of the joint spectral radius, and give a brief overview of how Theorem~\ref{counter} is subsequently deduced. The detailed structure of this paper is described at the end of the following section.

\section{Notation and statement of technical results}

Throughout this paper we denote the set of all $d \times d$ real matrices by $\Matd$. The symbol $\vvv\cdot\vvv$ will be used to denote the norm on $\Matd$ which is induced by the Euclidean norm on $\RR^d$, which satisfies $\vvv B\vvv=\rho(B^*B)^{1/2}$ for every $B \in  \Matd$. Other norms shall be denoted using the symbol $\| \cdot \|$. We shall say that a norm $\|\cdot\|$ on $\Mat$ is \emph{submultiplicative} if $\|AB\| \leq \|A\|\cdot\|B\|$ for all $A, B \in \Mat$. For the remainder of this paper we shall also denote $\varrho(\mathsf{A}_\alpha)$ simply by $\varrho(\alpha)$.

For the purposes of this paper we define a \emph{finite word}, or simply \emph{word} to be sequence $u=(u_i)$ belonging to $\{0,1\}^n$ for some integer $n \geq 0$. We will typically use $u$, $v$ or $w$ to represent finite words.  If $u \in \{0,1\}^n$ then we say that $u$ is length $n$, which we denote by $|u| = n$.  If $|u|$ is zero then the word $u$ is called \emph{empty}. The number of terms of $u$ which are equal to $1$ is denoted by $|u|_1$. If $u$ is nonempty, the quantity $|u|_1/|u|$ is called the \emph{$1$-ratio} of $u$ and is written $\varsigma(u)$. The two possible words of length one shall often be denoted simply by $0$ and $1$. We denote the set of all finite words by $\Omega$.

We will define an \emph{infinite word} to be a sequence $x = (x_i)$ belonging to $\{0,1\}^\NN$.  We will typically use $x$, $y$ or $z$ to represent infinite words.  If the word can be either finite or infinite, we will typically use $\omega$.  We denote the set of all infinite words by $\Sigma$, and define a metric $d$ on $\Sigma$ as follows.  Given $x,y \in \Sigma$ with $x=(x_i)_{i=1}^\infty$ and $y=(y_i)_{i=1}^\infty$, define $\mathfrak{n}(x,y):=\inf \{i \geq 1 \colon x_i \neq y_i\}$. We now define $d(x,y):=1/2^{\mathfrak{n}(x,y)}$ for all $x,y \in \Sigma$, where we interpret the symbol $1/2^\infty$ as being equal to zero. The topology on $\Sigma$ which is generated by the metric $d$ coincides with the infinite product topology on $\Sigma=\{0,1\}^{\mathbb{N}}$. In particular $\Sigma$ is compact and totally disconnected. For any nonempty finite word $u=(u_i)_{i=1}^n$ the set $\{x \in \Sigma \colon x_i=u_i \text{ for all }1 \leq i \leq n\}$ is both closed and open. Since every open ball in $\Sigma$ has this form for some $u$, the collection of all such sets generates the topology of $\Sigma$.

We define the \emph{shift transformation} $T \colon \Sigma \to \Sigma$ by $T[(x_i)_{i=1}^\infty]:=(x_{i+1})_{i=1}^\infty$. The shift transformation is continuous and surjective.  We define the projection $\pi_n : \Sigma \to \Omega$ by $\pi_n[(x_i)_{i=1}^\infty] =  (x_i)_{i=1}^n$.

If $u = u_1 u_2 \dots u_n$ and $v = v_1 v_2 \dots v_m$ are finite words, then we define the \emph{concatenation} of $u$ with $v$  as $u v = u_1 u_2 \dots u_n v_1 v_2 \dots v_m$, the finite word of length~$n + m$.  Note that if $u$ is the empty word then $uv=vu=v$ for every word $v$.  The set $\Omega$ endowed with the operation of concatenation is a semigroup.

Given a word $u$ and positive integer $n$ we let $u^n$ denote the linear concatenation of $n$ copies of $u$, so that for example $u^4:=uuuu$. If $u$ is a nonempty word of length $n$, we let $u^\infty$ denote the unique infinite word $x \in \Sigma$ such that $x_{kn+i}=u_i$ for all integers $i,k$ with $k \geq 1$ and $1 \leq i \leq n$. Clearly any infinite word $x \in \Sigma$ satisfies $T^nx=x$ for an integer $n \geq 1$ if and only if there exists a word $u$ such that $x=u^\infty$ and $|u|$ divides $n$.

If $u$ is a nonempty word, and $\omega$ is either a finite or infinite word, we say that $u$ is a \emph{subword} of $\omega$ if there exists an integer $k \geq 0$ such that $u_i=\omega_{k+i}$ for all integers $i$ in the range $1 \leq i \leq |u|$. We denote this relationship by $u \prec \omega$. Clearly $u \prec \omega$ if and only if there exist a possibly empty word $v \in \Omega$ and a finite or infinite word $\omega'$ such that $\omega = au\omega'$.  An infinite word $x$ is said to be \emph{recurrent} if every finite subword $u \prec x$ occurs as a subword of $x$ an infinite number of times.  A finite or infinite word $\omega$ is called \emph{balanced} if for every pair of finite subwords $u,v$ such that $u,v \prec \omega$ and $|u|=|v|$, we necessarily have $||u|_1-|v|_1| \leq 1$. Clearly $\omega$ is balanced if and only if every subword of $\omega$ is balanced. An infinite balanced word which is not eventually periodic is called \emph{Sturmian}.

The following standard result describes the principal properties of balanced infinite words which will be applied in this paper:

\begin{theorem}\label{Xphi}
If $x \in \Sigma$ is balanced then the limit $\varsigma(x):= \lim_{n \to \infty} \varsigma(\pi_n(x))$ exists. For each $\gamma \in [0,1]$, let $X_\gamma$ denote the set of all recurrent balanced infinite words $x \in \Sigma$ for which $\varsigma(x)=\gamma$. These sets have the following properties:
\begin{enumerate}[(i).]
\item
Each $X_\gamma$ is compact and nonempty.
\item
For each $\gamma \in [0,1]$, the restriction of $T$ to $X_\gamma$ is a continuous, minimal, uniquely ergodic transformation of $X_\gamma$. If $\mu$ is the unique ergodic probability measure supported in $X_\gamma$, then $\mu(\{x \colon x_1=1\})=\gamma$.
\item
If $\gamma = p/q \in [0,1] \cap \mathbb{Q}$ in lowest terms then the cardinality of $X_\gamma$ is $q$, and for each $x \in X_\gamma$ we have $X_\gamma = \{x,Tx,\ldots,T^{q-1}x\}$. If $\gamma \in [0,1] \setminus \mathbb{Q}$ then $X_\gamma$ is uncountably infinite.
\end{enumerate}
\end{theorem}

\begin{example}
We have $X_{2/5}=\{(00101)^\infty, (01010)^\infty, (10100)^\infty, (01001)^\infty, (10010)^\infty\}$.
\end{example}

Theorem~\ref{Xphi} does not appear to exist in the literature in the precise form given above, but it may be established without difficulty by combining various results from the second chapter of \cite{Lot}. The key step in obtaining Theorem~\ref{Xphi} is to show that $x \in X_\gamma$ if and only if there exists $\delta \in [0,1)$ such that either $x_n \equiv \lfloor (n+1)\gamma+\delta \rfloor  - \lfloor n\gamma + \delta\rfloor$, or $x_n \equiv \lceil (n+1)\gamma+\delta \rceil  - \lceil n\gamma + \delta\rceil$, see Lemmas~2.1.14 and 2.1.15 of \cite{Lot}. Once this identification has been made, the dynamical properties of $X_\gamma$ under the shift transformation largely follow from the properties of the rotation map $z \mapsto z+\gamma$ defined on $\mathbb{R}/\mathbb{Z}$.

Given a nonempty finite word $u=(u_i)^n_{i=1}$ and real number $\alpha \in [0,1]$, we put
\[\mathcal{A}^{(\alpha)}(u):= A^{(\alpha)}_{u_n}A^{(\alpha)}_{u_{n-1}}\cdots A^{(\alpha)}_{u_2}A^{(\alpha)}_{u_1}\]
and
\[\mathcal{A}(u):=A_{u_n}A_{u_{n-1}}\cdots A_{u_2}A_{u_1}=\mathcal{A}^{(1)}(u).\]
For every $x \in \Sigma$, $\alpha \in [0,1]$ and $n \geq 1$ we also define
\[\mathcal{A}^{(\alpha)}(x,n):=\mathcal{A}^{(\alpha)}(\pi_n(x)),\qquad \mathcal{A}(x,n):=\mathcal{A}(\pi_n(x))=\mathcal{A}^{(1)}(x,n).\]
Note that the function $\mathcal{A}(x,n)$ satisfies the cocycle relationship
\[\mathcal{A}(x,n+m) = \mathcal{A}(T^nx,m)\mathcal{A}(x,n)\]
for every $x \in \Sigma$, $n,m \geq 1$.

Our main task in proving Theorem~\ref{counter} is to characterise those infinite words $x \in \Sigma$ for which $\mathcal{A}(x,n)$ grows rapidly in terms of the sets $X_\gamma$. To do this we must be able to specify what is meant by rapid growth. Let us therefore say that an infinite word $x \in \Sigma$ is a \emph{strongly extremal} word for $\mathsf{A}_\alpha$ if there is a constant $\delta>0$ such that $\vvv\mathcal{A}^{(\alpha)}(x,n)\vvv \geq \delta \varrho(\alpha)^n$ for all $n\geq 1$, and \emph{weakly extremal} for $\mathsf{A}_\alpha$ if $\lim_{n\to\infty} \vvv\mathcal{A}^{(\alpha)}(x,n)\vvv^{1/n}=\varrho(\alpha)$. It is obvious that every strongly extremal word is also weakly extremal. Note also that since all norms on $\Mat$ are equivalent, these definitions are unaffected if another norm $\|\cdot\|$ is substituted for $\vvv\cdot \vvv$. We shall say that $\mathfrak{r} \in [0,1]$ is the \emph{unique optimal $1$-ratio} of $\mathsf{A}_\alpha$ if for every $x \in \Sigma$ which is weakly extremal for $\mathsf{A}_\alpha$ we have $\varsigma(\pi_n(x)) \to \mathfrak{r}$. Note that the existence of a unique optimal $1$-ratio is a nontrivial property, and is shown in Theorem~\ref{technical}.  For example, if $\mathsf{A}\subset \Mat$ is a pair of isometries then no unique optimal $1$-ratio for $\mathsf{A}$ exists.  It is not difficult to see that if $\mathsf{A}_\alpha$ has a unique optimal $1$-ratio which is irrational, then $\mathsf{A}_\alpha$ cannot satisfy the finiteness property, and it is this principle which underlies the present work as well as the work of Bousch-Mairesse \cite{BM} and Kozyakin \cite{Koz3}.

The principal technical result of this paper is the following theorem which allows us to relate all of the concepts defined so far in this section:
\begin{theorem}\label{technical}
There exists a continuous, non-decreasing surjection $\mathfrak{r} \colon [0,1] \to [0,\frac{1}{2}]$ such that for each $\alpha$, $\mathfrak{r}(\alpha)$ is the unique optimal $1$-ratio of $\mathsf{A}_\alpha$. For each $\alpha \in [0,1]$, every element of $X_{\mathfrak{r}(\alpha)}$ is strongly extremal for $\mathsf{A}_\alpha$. Moreover, for every compact set $K \subset (0,1]$ there exists a constant $C_K>1$ such that

\begin{equation}\label{niceformula}
C_K^{-1} \leq  \frac{\rho\left(\mathcal{A}^{(\alpha)}(x,n)\right)}{\varrho(\alpha)^n} \leq \frac{\Lvvv\mathcal{A}^{(\alpha)}(x,n)\Rvvv}{\varrho(\alpha)^n} \leq C_K
\end{equation}
whenever $\alpha \in K$, $x \in X_{\mathfrak{r}(\alpha)}$ and $n \geq 1$. Conversely, if $x \in \Sigma$ is a recurrent infinite word which is strongly extremal for $\mathsf{A}_\alpha$ then $x \in X_{\mathfrak{r}(\alpha)}$, and if $x \in \Sigma$ is any infinite word which is weakly extremal for $\mathsf{A}_\alpha$ then $(1/n)\sum_{k=0}^{n-1}\mathrm{dist}(T^kx,X_{\mathfrak{r}(\alpha)}) \to 0$.
\end{theorem}

\begin{remark} The definition of a strongly extremal infinite word is similar to the one previously proposed by V.~ S.~Kozyakin \cite{Koz2007}, whereas the definition of a weakly extremal infinite word is similar to a definition used previously by the fourth named author \cite{Theys}. In both instances the infinite word is simply referred to as `extremal'.
\end{remark}

\begin{remark} Note that balanced/Sturmian words (and measures) arise as optimal trajectories in various optimisation problems -- see, e.g., \cite{Bousch,BS,HO,Jenk1}.
\end{remark}

\begin{remark} A less general version of parts of Theorem \ref{technical}
    was proved in \cite{Theys}.
\end{remark}

The structure of the paper is as follows: Sections~\ref{sec3} and \ref{sec4} deal with important preliminaries, such as general properties of joint spectral radius and of balanced words. In Section~\ref{sec5} we show that every strongly extremal infinite word is balanced. In Section~\ref{section6} we introduce an important auxiliary function $S$ defined as the logarithm of the exponent growth of the norm of an arbitrary matrix product taken along balanced words with a fixed 1-ratio. In Section~\ref{sec7} we apply results from preceding sections to prove Theorem~\ref{technical}. Finally, in Section~\ref{sec8} we deduce Theorem~\ref{counter} from Theorem~\ref{technical}. Section~\ref{sec9} contains some open questions and conjectures.

We believe it is worth describing here briefly how Theorem~\ref{technical} leads to Theorem~\ref{counter}. Once we have established the existence of such a function $\mathfrak r$, we may take any irrational $\gamma$ and conclude that any element $\alpha$ of the preimage $\mathfrak r^{-1}(\gamma)$ is a counterexample to the finiteness conjecture (since any weakly extremal word must be aperiodic).

To construct a specific counterexample, we take $\gamma=\frac{3-\sqrt5}2$ and choose the \emph{Fibonacci word} $u_\infty$ as a strongly extremal word for this 1-ratio. Recall that $u_\infty=\lim_n u_{(n)}$, where $u_{(1)}=1, u_{(2)}=0$ and $u_{(n+1)}=u_{(n)}u_{(n-1)}$ for $n\ge2$. Now consider the morphism $h:\Omega\to\Mat$ such that $h(0)=A_0, h(1)=A_1$. Denote $B_n:=h(u_{(n)})$; we thus have $B_{n+1}=B_n B_{n-1}$. One can easily show that $\tr(B_n)=\tau_n$, the sequence described in Theorem~\ref{counter}. To obtain explicit formulae for $\alpha_*$, we show that the auxillary function $S$ introducted in Section \ref{section6} is differentiable at $\gamma = \frac{3-\sqrt{5}}{2}$ and that $-\log\alpha_*=S'\left(\frac{3-\sqrt5}2\right)$. We then compute this derivative, which will yield (\ref{eq:alphastar}).

\section{General properties of the joint spectral radius and extremal infinite words}\label{sec3}

We shall begin with some general results concerning the joint spectral radius. The following characterisation of the joint spectral radius will prove useful on a number of occasions:
\begin{lemma}\label{bowf}
Let $\alpha \in [0,1]$ and let $\|\cdot\|$ be any submultiplicative matrix norm. Then:
\[\varrho(\alpha) = \inf_{n \geq 1}\max\left\{\left\|\mathcal{A}(x,n)\right\|^{1/n}\colon x \in \Sigma\right\} =\sup_{n \geq 1}\max\left\{\rho\left(\mathcal{A}(x,n)\right)^{1/n}\colon x \in \Sigma\right\}.\]
\end{lemma}
\begin{proof}
We review some arguments from \cite{BW,DL}. Fix $\alpha \in [0,1]$ and a matrix norm $\|\cdot\|$, and define
\[\varrho_n^+(\alpha,\|\cdot\|):= \max\left\{\left\|A^{(\alpha)}_{i_1}\cdots A^{(\alpha)}_{i_n}\right\|\colon (i_1,\ldots,i_n)\in\{0,1\}^n\right\} = \max \left\{\left\|A^{(\alpha)}(x,n)\right\| \colon x \in \Sigma\right\}\]
and
\[\varrho_n^-(\alpha):=\max\left\{\rho\left(A^{(\alpha)}_{i_1}\cdots A^{(\alpha)}_{i_n}\right)\colon i_1,\ldots,i_n\in\{0,1\}\right\} = \max \left\{\rho\left(A^{(\alpha)}(x,n) \right)\colon x \in \Sigma\right\}.\]
Clearly each $\varrho_n^+(\alpha,\|\cdot\|)$ is nonzero, and $\varrho_{n+m}^+(\alpha,\|\cdot\|) \leq \varrho_{n}^+(\alpha,\|\cdot\|)\varrho_{m}^+(\alpha,\|\cdot\|)$ for every $n, m \geq 1$. Applying Fekete's subadditivity lemma \cite{Feck} to the sequence $\log \varrho_n^+(\alpha,\|\cdot\|)$ we obtain
\[\lim_{n \to \infty} \varrho^+_n(\alpha,\|\cdot\|)^{1/n} = \inf_{n \geq 1} \varrho^+_n(\alpha,\|\cdot\|)^{1/n}.\]
In particular the limit superior in the definition of $\varrho(\alpha)$ is in fact a limit. A well-known result of Berger and Wang \cite{BW} implies that
\[\lim_{n \to \infty} \varrho^+_n(\alpha,\|\cdot\|)^{1/n} = \limsup_{n \to \infty} \varrho^-_n(\alpha)^{1/n},\]
which in particular implies that the value $\varrho(\alpha)$ is independent of the choice of norm $\|\cdot\|$. Finally, note that if $\rho\left(A_{i_1}^{(\alpha)}\cdots A_{i_n}^{(\alpha)}\right)=\varrho^-_n(\alpha)$ for some $n$, then $\varrho_{nm}^-(\alpha) \geq \rho((A_{i_1}\cdots A_{i_n})^m) = \varrho_n^-(\alpha)^m$ for each $m \geq 1$, and hence the limit superior above is also a supremum.
\end{proof}
We may immediately deduce the following corollary, which was originally noted by C.~Heil and G.~Strang \cite{HS}:
\begin{lemma}\label{rhocts}
The function $\varrho \colon [0,1] \to\mathbb{R}$ is continuous.
\end{lemma}
\begin{proof}
The first of the two identities given in Lemma~\ref{bowf} shows that $\varrho$ is equal to the pointwise infimum of a family of continuous functions, and hence is upper semi-continuous. The second identity shows that $\varrho$ also equals the pointwise supremum of a family of continuous functions, and hence is lower semi-continuous.
\end{proof}

\begin{lemma}\label{extnorm}
For each $\alpha\in(0,1]$ there exists a matrix norm $\|\cdot\|_\alpha$ such that $\left\|A_i^{(\alpha)}\right\|_\alpha \leq \varrho(\alpha)$ for $i=0,1$. The matrix norms $\|\cdot\|_\alpha$ may be chosen so that the following additional property is satisfied: for every compact set $K \subset (0,1]$ there exists a constant $M_K>1$ such that $M_K^{-1}\|B\|_\alpha \leq \vvv B\vvv \leq M_K \|B\|_\alpha$ for all $B \in \Mat$ and all $\alpha \in K$.
\end{lemma}
\begin{proof}
Let $\mathsf{B}=\{B_1,\ldots,B_r\}$ be any finite set of $d \times d$ real matrices and let $\varrho(\mathsf{B})$ be its joint spectral radius. We say that $\mathsf{B}$ is \emph{irreducible} if the only linear subspaces $V \subseteq\RR^d$ such that $B_iV \subseteq V$ for every $i$ are $\{0\}$ and $\RR^d$. A classic theorem of N.~E.~Barabanov \cite{Ba} shows that if $\mathsf{B}$ is irreducible then there exists a constant $M_{\mathsf{B}}>1$ such that for each $n \geq 1$,
\[
\max\{\|B_{i_1}\dots B_{i_n}\| \colon i_j \in \{1,\ldots,r\}\} \leq M_{\mathsf{B}}\varrho(\mathsf{B})^n.
\]
Note in particular that necessarily $\varrho(\mathsf{B})>0$. It is then straightforward to see that if we define for each $v \in \mathbb{R}^d$
\[
\|v\|_{\mathsf{B}}:= \sup_{n \geq 0}\left\{\varrho(\mathsf{B})^{-n}\max \vvv B_{i_1}\cdots B_{i_n}v\vvv \colon i_j \in \{1,\ldots,r\}\right\},
\]
where $\vvv\cdot\vvv$ denotes the Euclidean norm, then $\|\cdot\|_{\mathsf{B}}$ is a norm on $\mathbb{R}^d$ which satisfies $\|B_i v\|_{\mathsf{B}} \leq \varrho(\mathsf{B})\|v\|_{\mathsf{B}}$ for every $i \in \{1,\ldots,r\}$ and $v \in \mathbb{R}^d$.  It follows that the operator norm on $\Mat$ induced by $\|\cdot\|_{\mathsf{B}}$ has the property $\|B_i\|_{\mathsf{B}} \leq \varrho(\mathsf{B})$ for each $B_i$.  More recent results due to F.~Wirth \cite[Thm. 4.1]{Wirth} and V.~S.~Kozyakin \cite{Kozmore} show that the constants $M_{\mathsf{B}}$ may be chosen so as to depend continuously on the set of matrices $\mathsf{B}$, subject to the condition that the perturbed matrix families also do not have invariant subspaces. It is easily shown that $\mathsf{A}_\alpha$ is irreducible for every $\alpha \in (0,1]$ and so the lemma  follows from these general results.
\end{proof}

We immediately obtain the following:
\begin{lemma}\label{gtr1}
For each $\alpha \in (0,1]$ we have $\varrho(\alpha)>1$.
\end{lemma}
\begin{proof}
Assume $\varrho(\alpha) \leq 1$ for some $\alpha \in (0,1]$. Then we have $\sup\{\|A^{(\alpha)}\left(0^n\right)\|_\alpha\colon n \geq 1\} \leq 1$ by Lemma~\ref{extnorm} and consequently $\sup\{\vvv A^n_0\vvv \colon n \geq 1\}<\infty$. Since $A_0^n=\left(\begin{smallmatrix} 1 & n \\ 0 & 1\end{smallmatrix}\right)$, we have $\lim_{n \to \infty}\vvv A_0^n\vvv =+ \infty$ and therefore we must have $\varrho(\alpha)>1$.
\end{proof}

Fix some norm $\|\cdot\|_\alpha$ which satisfies the conditions of Lemma~\ref{extnorm}. The following key result is a variation on part of \cite[Thm~2.2]{QBWF}. We include a proof here for the sake of completeness.

\begin{lemma}\label{Zset}
For each $\alpha \in (0,1]$ define
\[Z_\alpha:=\bigcap_{n=1}^\infty \left\{x \in \Sigma \colon \left\|\mathcal{A}^{(\alpha)}(x,n)\right\|_\alpha = \varrho(\alpha)^n\right\}.\]
Then each $Z_\alpha$ is compact and nonempty, and satisfies $TZ_\alpha \subseteq Z_\alpha$.
\end{lemma}
\begin{proof}
Fix $\alpha \in (0,1]$ and define for each $n\geq 1$
\[Z_{\alpha,n}:= \left\{x \in \Sigma \colon \left\|\mathcal{A}^{(\alpha)}(x,n)\right\|_\alpha = \varrho(\alpha)^n\right\}.\]
Clearly each $Z_{\alpha,n}$ is closed. If some $Z_{\alpha,n}$ were to be empty, then by Lemma~\ref{extnorm} we would have $\sup\left\{\left\|\mathcal{A}^{(\alpha)}(x,n)\right\|_\alpha \colon x \in \Sigma\right\} < \varrho(\alpha)^n$, contradicting Lemma~\ref{bowf}. For each  $n \geq 1$ we have $Z_{\alpha,n+1}\subseteq Z_{\alpha,n}$, since if $x \in Z_{\alpha,n+1}$ then
\begin{align*}
\varrho(\alpha)^{n+1}= \left\|\mathcal{A}^{(\alpha)}(x,n+1)\right\|_\alpha &\leq \left\|\mathcal{A}^{(\alpha)}(T^nx,1)\right\|_\alpha \left\|\mathcal{A}^{(\alpha)}(x,n)\right\|_\alpha \\&\leq \varrho(\alpha)
\left\|\mathcal{A}^{(\alpha)}(x,n)\right\|_\alpha \leq \varrho(\alpha)^{n+1}
\end{align*}
using Lemma~\ref{extnorm} and it follows that $x \in Z_{\alpha,n}$ also. We deduce that the set $Z_\alpha = \bigcap_{n=1}^\infty Z_{\alpha,n}$ is nonempty. Since each $Z_{\alpha,n}$ is closed, $Z_\alpha$ is closed and hence is compact. Finally, if $x \in Z_{\alpha,n+1}$ then we also have
\begin{align*}
\varrho(\alpha)^{n+1}= \left\|\mathcal{A}^{(\alpha)}(x,n+1)\right\|_\alpha &\leq \left\|\mathcal{A}^{(\alpha)}(Tx,n)\right\|_\alpha \left\|\mathcal{A}^{(\alpha)}(x,1)\right\|_\alpha \\&\leq \varrho(\alpha)\left\|\mathcal{A}^{(\alpha)}(Tx,n)\right\|_\alpha \leq \varrho(\alpha)^{n+1}
\end{align*}
so that $Tx \in Z_{\alpha,n}$, and we deduce from this that $TZ_\alpha \subseteq Z_\alpha$.
\end{proof}

The remaining lemmas in this section will be applied in the proof of Theorem~\ref{technical} to characterise the extremal orbits of $\mathsf{A}_\alpha$.

\begin{lemma}\label{strongex}
Let $\alpha \in (0,1]$ and $x\in\Sigma$. If $x$ is recurrent and strongly extremal for $\mathsf A_\alpha$, then $x \in Z_\alpha$.
\end{lemma}
\begin{proof}
Let $\alpha \in (0,1]$ and $x\in\Sigma \setminus Z_\alpha$, and suppose that $x$ is recurrent. We shall show that $\liminf_{n \to \infty} \varrho(\alpha)^{-n}\left\|\mathcal{A}^{(\alpha)}(x,n)\right\|_\alpha = 0$ and therefore $x$ is not strongly extremal, which proves the lemma. Since $x \notin Z_\alpha$, there exist $\varepsilon>0$ and $n_0 \geq 1$ such that $\left\|\mathcal{A}^{(\alpha)}(x,n_0)\right\|_\alpha <(1-\varepsilon)\varrho(\alpha)^{n_0}$. Since $x$ is recurrent, it follows that for each $k \geq 1$ we may find integers $r_k>r_{k-1}>\ldots >r_2>r_1=0$ such that $\left\|\mathcal{A}^{(\alpha)}(T^{r_i}x,n_0)\right\|_\alpha <(1-\varepsilon)\varrho(\alpha)^{n_0}$ for each $i$. By increasing $k$ and passing to a subsequence if necessary, it is clear that we may assume additionally that $r_{i+1}>r_i+n_0$ for $1 \leq i < k$. Define also $r_{k+1}:=r_k+n_0+1$. We have
\begin{align*}
\left\|\mathcal{A}^{(\alpha)}(x,r_{k+1})\right\|_\alpha & \leq \prod_{i=1}^{k} \left\|\mathcal{A}^{(\alpha)}(T^{r_i}x,n_0)\right\|_\alpha \left\|\mathcal{A}^{(\alpha)}(T^{r_i+n_0}x,r_{i+1}-r_i-n_0)\right\|_\alpha \\
&\leq (1-\varepsilon)^k \varrho(\alpha)^{r_{k+1}},
\end{align*}
and since $k$ may be taken arbitrarily large we conclude that
$$
\liminf_{n \to \infty} \varrho(\alpha)^{-n}\left\|\mathcal{A}^{(\alpha)}(x,n)\right\|_\alpha = 0,
$$
as desired.
\end{proof}
The following lemma is a straightforward corollary of a more general result due to S. J. Schreiber \cite[Lemma 1]{Sch}:
\begin{lemma}\label{schr}
Let $(f_n)$ be a sequence of continuous functions from $\Sigma$ to $\mathbb{R}$ such that $f_{n+m}(x) \leq f_n(T^mx)+f_m(x)$ for all $x \in \Sigma$ and $n,m \geq 1$. Then for each $x \in \Sigma$ and $m \geq 1$,
\[\liminf_{n \to \infty}\frac{1}{nm}\sum_{k=0}^{n-1}f_m(T^kx) \geq \liminf_{n \to \infty} \frac{1}{n}f_n(x).\]
\end{lemma}
\begin{lemma}\label{weakex}
Let $\alpha \in (0,1]$ and suppose that the restriction of $T$ to $Z_\alpha$ is uniquely ergodic, with $\mu$ being its unique $T$-invariant Borel probability measure. Then $\mathfrak{r}:=\mu\left(\{x \in \Sigma \colon x_1 =1\}\right)$ is the unique optimal $1$-ratio of $\mathsf{A}_\alpha$, and if $x \in \Sigma$ is weakly extremal, then
\[
\lim_{n \to \infty}\frac1n\sum_{k=0}^{n-1} \mathrm{dist}\left(T^kx,\supp \mu\right) =0.
\]
\end{lemma}
\begin{proof}

Let $\mathcal{M}$ denote the set of all Borel probability measures on $\Sigma$ equipped with the weak-* topology, which is defined to be the smallest topology such that $\mu \mapsto \int f\,d\mu$ is continuous for every continuous function $f \colon \Sigma \to \mathbb{R}$. This topology makes $\mathcal{M}$ a compact metrisable space \cite[Thm. II.6.4]{Parth}. Let us fix $\alpha \in (0,1]$ and suppose that $x \in \Sigma$ is weakly extremal. For each $n \geq 1$ define $\mu_n:=(1/n)\sum_{k=0}^{n-1}\delta_{T^kx} \in \mathcal{M}$, where $\delta_z \in \mathcal{M}$ denotes the Dirac probability measure concentrated at $z \in \Sigma$. We claim that $\lim_{n \to \infty} \mu_n = \mu$ in the weak-* topology.

Applying Lemma \ref{schr} with $f_n(x):=\log \left\|\mathcal{A}^{(\alpha)}\left(x,n\right)\right\|_\alpha$ and noting that $f_n(x) \leq n\log\varrho(\alpha)$ for all $x$ and $n$, we obtain
\begin{equation}\label{calc0}\lim_{n \to \infty} \int \frac{1}{N}  \log \left\|\mathcal{A}^{(\alpha)}\left(z,N\right)\right\|_\alpha d\mu_n(z)=\lim_{n \to \infty}\frac{1}{nN}\sum_{i=0}^{n - 1}\log\left\|\mathcal{A}^{(\alpha)}\left(T^ix,N\right)\right\|_\alpha =\log \varrho(\alpha)\end{equation}
for every $N \geq 1$. As in the proof of Lemma~\ref{Zset} we let $Z_{\alpha,N} = \{z \in \Sigma \colon \|\mathcal{A}^{(\alpha)}(z,N)\|_\alpha=\varrho(\alpha)^N\}$ for each $N \geq 1$, and we recall that $Z_{\alpha,N+1}\subseteq Z_{\alpha,N}$ for every $N$. Let $\nu \in \mathcal{M}$ be any limit point of the sequence $(\mu_n)$. If $f \colon \Sigma \to \mathbb{R}$ is any continuous function then it follows easily from the definition of $(\mu_n)$ that $|\int f\,d\nu - \int f\circ T \,d\nu| \leq \limsup_{n \to \infty} |\int f\circ T  d\mu_n - \int f\,d\mu_n| =0$ and it follows that $\nu$ is $T$-invariant. For each $N \geq 1$ we have
\[ \int \frac{1}{N}  \log \left\|\mathcal{A}^{(\alpha)}(z,N)\right\|_\alpha d\nu(z)=\log \varrho(\alpha),\]
and since $\left\|\mathcal{A}^{(\alpha)}(z,N) \right\|_\alpha\leq \varrho(\alpha)^N$ for all $z \in \Sigma$ it follows from this that $\nu\left(Z_{\alpha,N}\right)=1$. Since this applies for every $N$, and $Z_{\alpha,N+1} \subseteq Z_{\alpha,N}$ for every $N$, we deduce that $\nu(Z_\alpha)=1$. By hypothesis $\mu$ is the unique $T$-invariant element of $\mathcal{M}$ giving full measure to $Z_\alpha$, and it follows that $\nu = \mu$. We have shown that $\mu$ is the only weak-* accumulation point of the sequence $(\mu_n)$, and since $\mathcal{M}$ is compact and metrisable we deduce that $\lim_{n \to \infty}\mu_n=\mu$, which completes the proof of the claim.

The proof of the lemma now follows easily. Let $f \colon \Sigma \to \mathbb{R}$ be the characteristic function of the set $\{x \in \Sigma \colon x_1=1\}$, and note that $f$ is continuous since this set is both open and closed. Define a further continuous function by $g(x):=\mathrm{dist}(x,\supp \mu)$. Since $\mu_n \to \mu$ we may easily derive
\[\lim_{n \to \infty} \varsigma(\pi_n(x)) = \lim_{n \to \infty} \frac{1}{n}\sum_{i=0}^{n-1}f\left(T^ix\right) = \lim_{n \to \infty} \int f\,d\mu_n = \int f\,d\mu = \mu(\{x \in \Sigma \colon x_1=1\})=\mathfrak{r}\]
and
\[\lim_{n \to \infty} \frac{1}{n}\sum_{i=0}^{n-1}\mathrm{dist}(T^ix,\supp \mu) = \lim_{n \to \infty} \int g\,d\mu_n = \int g\,d\mu =0\]
as required. The proof is complete.

\end{proof}

\section{General properties of balanced words}\label{sec4}

In this short and mostly expository section we present some combinatorial properties of balanced words which will be applied in subsequent sections. We first require some additional definitions.

Given two nonempty finite words $u,v$ of equal length, we write $u < v$ if $u$ strictly precedes  $v$ in the lexicographical order: that is, $u < v$ if and only if there is $k \geq 1$ such that $u_k=0$, $v_k=1$, and $u_i=v_i$ when $1 \leq i <k$. We define the \emph{reverse} of a finite word $u$, which we denote by $\tilde u$, to be the word obtained by listing the terms of $u$ in reverse order.   That is, if $u = u_1 u_2 \cdots u_n$ then $\tilde u = u_n u_{n-1} \cdots u_1$.  We say that a finite word $p$ is a \emph{palindrome} if $\tilde p = p$. Since the reverse of the empty word is also the empty word, the empty word is a palindrome. We say that two finite words $u$ and $v$ of equal length are \emph{cyclic permutations} of each other, and write $u \simeq v$, if there exist finite words $a$ and $b$ such that $u=ab$ and $v=ba$. For each $n\geq 0$ this defines an equivalence relation on the set of words of length $n$.

We begin by collecting together some standard results from  \cite{Lot}:
\begin{lemma}\label{nolongwords}
Let $\gamma \in (0,1)$ and $x \in X_\gamma$, and choose any $N > \max\{\lceil \gamma^{-1}\rceil, \lceil (1-\gamma)^{-1}\rceil\}$. Then neither $0^N$ nor $1^N$ is a subword of $x$.
\end{lemma}
\begin{proof}
Let $u \prec x$ with $|u|=N$. By \cite[Prop. 2.1.10]{Lot} we have $\gamma|u| +  1 \geq |u|_1 \geq \gamma |u| -1$. In particular we have $|u|_1 > \gamma \lceil \gamma^{-1}\rceil - 1 \geq 0$ and $|u|-|u|_1 > (1-\gamma)\lceil (1-\gamma)^{-1}\rceil -1\geq 0$, so $0<|u|_1<|u|$ and $u$ cannot be equal to $0^N$ or $1^N$.
\end{proof}
\begin{definition}
Let $\mathcal{W} \subset \Omega \times \Omega$ be the smallest set with the following two properties: $(0,1) \in \mathcal{W}$; if $(u,v)\in \mathcal{W}$, then $(uv,v) \in \mathcal{W}$ and $(u,vu)\in\mathcal{W}$. We say that $u \in \Omega$ is a \emph{standard word} if either $(v,u) \in \mathcal{W}$ or $(u,v)\in \mathcal{W}$ for some $v \in \Omega$.
\end{definition}
\begin{lemma}\label{standard}
The set of standard words has the following properties:
\begin{enumerate}[(i)]
\item
If $u$ is standard, with $|u|=q$ and $|u|_1=p$, then $u^\infty \in X_{p/q}$.
\item
For every $\gamma \in [0,1]$ there exists $x \in X_\gamma$ such that for infinitely many $q \in \mathbb{N}$ the word $\pi_q(x)$ is standard.
\end{enumerate}
\end{lemma}
\begin{proof}
(i). If $q=1$ then the result is trivial. For $q>1$, \cite[Prop. 2.2.15]{Lot} shows that every standard word is balanced. If $u$ is standard, then it is clear from the definition that $u^n$ is a subword of a standard word for every $n \geq 1$. In particular every $u^n$ is balanced and therefore $u^\infty$ is balanced.

(ii). Let $x$ be the infinite word defined by $x_n:= \lfloor \gamma(n+2)\rfloor - \lfloor \gamma(n+1)\rfloor \in \{0,1\}$ for all $n \geq 1$. This word is called the \emph{characteristic word} for $\gamma$. It is shown in \cite[Prop 2.2.15]{Lot} that $x$ has the required properties.
\end{proof}
The following result is given in the proof of \cite[Prop. 2.1.3]{Lot}. Note that $p$ may be the empty word; for example, this is true in the case $w = 0011$.
\begin{lemma}\label{lothaire1}
Let $w$ be a finite word which is not balanced, let $u$ and $v$ be subwords of $w$ of equal length such that $|u|_1 \geq 2+|v|_1$, and suppose that $u,v$ have the minimum possible length for which this property may be satisfied. Then there is a palindrome $p$ such that $u=1p 1$ and $v = 0 p 0$.
\end{lemma}
The following two results arise in the fourth named author's PhD thesis \cite{Theys}:
\begin{lemma}\label{theys1}
Let $w$ be a finite word and $p$ a palindrome, and suppose that $0p0$ and $1p1$ are subwords of $w$. Then there is a finite word $b$, which may be empty, such that either $0p0b1p1$ or $1p1b0p0$ is a subword of $w$.
\end{lemma}
\begin{proof}
Recall that $u\prec v$ means that $u$ is a subword of $v$.
Since $0p0$ and $1p1$ are both subwords of $w$, the only alternative is that they occur in an overlapping manner: that is, there are finite words $d,e,f$ such that $0d1e0f1 \prec w$, where $d1e=e0f=p$, or similarly with $0$ and $1$ interchanged. Since $\tilde p = p$, the relation $d1e=e0f=p$ implies $\tilde e 1 \tilde d = e0f$, and since $|\tilde e | = |e|$ we obtain $1=0$, a contradiction. We conclude that the words $0p0$ and $1p1$ cannot overlap, and the result follows.
\end{proof}

\begin{lemma}\label{theys2}
Let $u$ be a finite word which is not balanced. Then there exist words $a,w,b$ such that $a w b \prec u$ and one of the following two possibilities holds: either $\tilde b > a$ and $\tilde w > w$, or $\tilde a > b$ and $w > \tilde w$.
\end{lemma}
\begin{proof}
Combining Lemmas \ref{lothaire1} and \ref{theys1} we find that there exist words $p,v$ such that $\tilde p = p$ and either $0p0v1p1 \prec u$, or $1p1v0p0 \prec u$. In the former case we may take $a:=0p$, $b:=p1$ and $w:=0v1$, and in the latter case we may take $a:=1p$, $b:=p0$ and $w:=1v0$.
\end{proof}

Finally, we require the following lemma which characterises those finite words for which all cyclic permutations are balanced. This result appears to be something of a ``folklore theorem'' in the theory of balanced words; to the best of our knowledge, the proof which we present here is original. A version of this result appears as \cite[Thm 6.9]{Ret}.  Note that the word $u:=1001$ is an example of a balanced word with the property that $u^\infty$ is not balanced.

\begin{lemma}\label{cyclic-balanced}
Let $u$ be a nonempty finite word. Then the following are equivalent:
\begin{enumerate}[(i)]
\item
Every cyclic permutation of $u$ is balanced.
\item
The finite word $u^2$ is balanced.
\item
The infinite word $u^\infty$ is balanced.
\end{enumerate}
\end{lemma}
\begin{proof}
It is clear that (iii)$\implies$(ii)$\implies$(i). To prove the implication (i)$\implies$(ii) by we shall show that if $u$ is a nonempty finite word such that $u^2$ is not balanced, then there is a cyclic permutation of $u$ which is not balanced.

Let us then suppose that $u$ is a finite nonempty word such that $u^2$ is  not balanced. Let $a,b$ be subwords of $u^2$ of equal length such that $||a|_1-|b|_1|\geq 2$, and suppose that no pair of shorter subwords may be found which also has this property. Clearly we have $||a|_1-|b|_1|=2$, and without loss of generality we shall assume that $|a|_1=2+|b|_1$. By Lemma~\ref{lothaire1} there exists a palindrome $p$ such that $a=1p1$ and $b=0p0$, and it follows from Lemma~\ref{theys1} that $|a|,|b| \leq |u|$. We may therefore choose words $c$ and $d$ such that $|c|=|d| = |u|-|a|=|u|-|b|$ and $ac \simeq bd \simeq u$. Since $|ac|_1=|bd|_1=|u|_1$ we have $|d|_1=2+|c|_1$, and since $a$ and $b$ are the shortest words with this property we must have $|b|=|a|\leq |c|$. Now, since $ac \simeq u$, it is not difficult to see that every word which is a subword of some cyclic permutation of $u$ and has length at most $|c|$ must occur as a subword of the word $cac$. In particular $b \prec cac$, and since $|b|=|a|$ we have either $b \prec ca$ or $b \prec ac$. In either case we have shown that there exists a cyclic permutation of $u$ which has both $a$ and $b$ as subwords, and no word with that property may be balanced. We conclude that (i) cannot not hold when (ii) does not hold, and so (i)$\implies$(ii) as required.

It is now straightforward to show that (ii)$\implies$(iii). Let $u$ be a finite nonempty word such that $u^2$ is balanced; then every cyclic permutation of $u$ is balanced, since the cyclic permutations of $u$ are precisely the subwords of $u^2$ with length $|u|$. Now, the cyclic permutations of $u^2$ are precisely the words of the form $v^2$ where $v \simeq u$; but since (i)$\implies$(ii), all of these cyclic permutations must be balanced also. Applying the implication (i)$\implies$(ii) again we deduce that $u^4$ is balanced. Repeating this procedure inductively shows that $u^{2^k}$ is balanced for every $k \geq 1$, and this yields (iii).\end{proof}

\section{Relationships between balanced words and extremal orbits}\label{sec5}

The principal goal of this section is to show that for each $\alpha \in (0,1]$, every recurrent $x \in \Sigma$ which is strongly extremal for $\mathsf{A}_\alpha$ is balanced. We also prove some related ancillary results which will be applied in the following section.

The following valuable lemma shows that under quite mild conditions the trace, spectral radius, Euclidean norm and smallest diagonal element of a matrix of the form $\mathcal{A}(u)$ approximate each other quite closely. For every $B \in \Mat$ we define $\mathfrak{d}(B)$ to be the minimum modulus of the diagonal entries of $B$.
\begin{lemma}\label{rho-norm}
Let $\alpha \in [0,1]$ and $N \geq 2$, and let $u$ be a nonempty finite word such that $0^N, 1^N \nprec u$. Then,
\[\frac{1}{2N^2}\Lvvv \mathcal{A}^{(\alpha)}(u)\Rvvv \leq \mathfrak{d}\left(\mathcal{A}^{(\alpha)}(u)\right) \leq \frac{1}{2}\tr \mathcal{A}^{(\alpha)}(u) \leq \rho\left(\mathcal{A}^{(\alpha)}(u)\right) \leq \Lvvv\mathcal{A}^{(\alpha)}(u)\Rvvv.\]
\end{lemma}
\begin{proof}
Let $\mathfrak{m}(B)$ denote the maximum of the entries of a non-negative matrix $B \in \Mat$. The inequalities
\[\vvv B\vvv = \sqrt{\rho(B^*B)} \leq \sqrt{\tr (B^*B)} \leq 2\mathfrak{m}(B)\]
and
\[\mathfrak{d}(B) \leq \frac{1}{2}\tr B \leq \rho(B) \leq \vvv B\vvv\]
are elementary. To prove the lemma, it therefore suffices to show that $\mathfrak{m}\left(\mathcal{A}^{(\alpha)}(u)\right) \leq N^2 \mathfrak{d}(\mathcal{A}^{(\alpha)}(u))$ whenever $0^N, 1^N \nprec u$. Since $\mathcal{A}^{(\alpha)}(u)\equiv \alpha^{|u|_1}\mathcal{A}(u)$ it is clearly sufficient to consider only the case $\alpha=1$.

Let us prove this inequality. We shall suppose that the final symbol occurring in $u$ is $0$, since the opposite case is easily dealt with by symmetry. Let $n \geq 1$ and $a_1,\ldots,a_n \geq 1$ be integers such that either $u=0^{a_n}1^{a_{n-1}}0^{a_{n-2}}\cdots1^{a_2}0^{a_1}$ with $n$ odd, or $u=1^{a_n}0^{a_{n-1}}1^{a_{n-2}}\cdots1^{a_2}0^{a_1}$ with $n$ even. By hypothesis we have $a_k \leq N-1$ for every $k$.

For $1 \leq k \leq n$ let us define
 \[\frac{p_k}{q_k}:=
\cfrac{1}{a_1+
\cfrac{1}{a_2+\dotsb+
\cfrac{1}{a_{k-1}+
\cfrac{1}{a_k}
}}}
\]
in least terms, and define also $p_0,q_{-1}:=0$ and $p_{-1},q_0:=1$. The integers $p_k, q_k$ then satisfy the recurrence relations $p_{k}=a_{k}p_{k-1}+p_{k-2}$ and $q_{k}=a_{k}q_{k-1}+q_{k-2}$ for all $k$ in the range $1\leq k\leq n$. A well-known formula for $p_k/q_k$ implies
\[
\mathcal{A}(u)=A_0^{a_n}A_1^{a_{n-1}}\cdots A_0^{a_1} = \left(\begin{array}{cc}1&a_n\\0&1\end{array}\right) \left(\begin{array}{cc}1&0\\a_{n-1}&1\end{array}\right)\cdots \left(\begin{array}{cc}1&a_1\\0&1\end{array}\right)= \left(\begin{array}{cc}p_n&q_n\\p_{n-1}&q_{n-1}\end{array} \right)\]
if $n$ is odd, and
\[
\mathcal{A}(u)=A_1^{a_n}A_0^{a_{n-1}}\cdots A_0^{a_1} = \left(\begin{array}{cc}1&0\\a_n&1\end{array}\right)
\left(\begin{array}{cc}1&a_{n-1}\\0&1\end{array}\right)\cdots \left(\begin{array}{cc}1&a_1\\0&1\end{array}\right)= \left(\begin{array}{cc}p_{n-1}&q_{n-1}\\p_{n}&q_{n} \end{array}\right)
\]
if $n$ is even (see, e.g., \cite{Frame}). If $n$ is odd then clearly $\mathfrak{d}(\mathcal{A}(u))=\min\{p_n,q_{n-1}\}$, and since $q_n=a_{n}q_{n-1}+q_{n-2} \leq (a_{n}+1)q_{n-1} \leq Nq_{n-1}$ and $p_{n}/q_{n} \geq 1/(a_1+1) \geq 1/N$ we obtain $\mathfrak{m}(\mathcal{A}(u))=q_n \leq \min\{Np_n,Nq_{n-1}\}< N^2 \mathfrak{d}(\mathcal{A}(u))$ as required. If $n$ is even then similarly $\mathfrak{m}(\mathcal{A}(u))=q_n \leq N q_{n-1} \leq N^2 p_{n-1}=N^2\mathfrak{d}(\mathcal{A}(u))$. The proof is complete.
\end{proof}

Let $a,w,b$ be nonempty finite words with $|a|=|b|$. We shall say that $(a,w,b)$ is a \emph{suboptimal triple} if either $\tilde a > b$ and $w > \tilde w$, or $\tilde b > a$ and $\tilde w > w$.
We require the following lemma due to V. Blondel, J. Theys and A. Vladimirov \cite[Lemma 4.2]{BTV}:
\begin{lemma}\label{beeteevee}
Let $w$ be a nonempty finite word. Then $\mathcal{A}(\tilde w) - \mathcal{A}(w) = k(w)J$, where $k(w) \in \mathbb{Z}$ and
\[J := A_0A_1 - A_1A_0 = \left(\begin{array}{cc}1&0\\0&-1\end{array}\right).\]
Moreover, $k(w)$ is positive if and only if $w > \tilde w$, and negative if and only if $w < \tilde w$.
\end{lemma}
The following is a slightly strengthened version of \cite[Lemma~4.3]{BTV}:

\begin{lemma}\label{sot2}
Let $(a,w,b)$ be a suboptimal triple, let $B_1,B_2$ be non-negative matrices, and let $\alpha \in [0,1]$. Then
\[\tr\left(B_1 \mathcal{A}^{(\alpha)}(a \tilde w b)B_2\right)\geq \tr\left(B_1\mathcal{A}^{(\alpha)}(a w b)B_2\right) + \alpha^{|a w b|_1}\mathfrak{d}(B_1)\mathfrak{d}(B_2).\]
\end{lemma}
\begin{proof}
Since $\tr \mathcal{A}^{(\alpha)}(u)=\alpha^{|u|_1}\tr \mathcal{A}(u)$ for every finite word $u$ it is clearly sufficient to treat only the case $\alpha=1$. We shall deal first with the case where $\tilde a >b$ and $w > \tilde w$, the alternative case being similar. Since $\tilde a > b$ we may write $a=u 1 c$, $b=\tilde c 0 \tilde v$ for some finite words $c$, $u$ and $v$ (which may be empty). Note that $J$ satisfies the relations
\[A_1JA_1= A_0JA_0 = J,\qquad A_0JA_1 = \left(\begin{array}{cc}0 & -1\\ -1 & -1\end{array}\right),\qquad A_1JA_0 = \left(\begin{array}{cc}1 & 1\\ 1 & 0\end{array}\right),\]
and hence by Lemma~\ref{beeteevee},
\[\tr\left(\mathcal{A}(a)(\mathcal{A}(\tilde w) - \mathcal{A}(w))\mathcal{A}(b)\right)=k(w) \tr\left(\mathcal{A}(u)\left(\begin{array}{cc}1 & 1\\ 1 & 0\end{array}\right)\mathcal{A}(\tilde v)\right)  \geq 1.\]
Now, a direct calculation shows that for any non-negative matrix $C \in \mathbf{M}_2(\mathbb{R})$ we have $\tr(B_1CB_2) \geq \mathfrak{d}(B_1)\mathfrak{d}(B_2)\tr(C)$. Since the matrix $\mathcal{A}(a)(\mathcal{A}(\tilde w) - \mathcal{A}(w))\mathcal{A}(b)$ is non-negative, we deduce that
\begin{align*} \tr(B_1 \mathcal{A}(a \tilde w b)B_2)-  \tr(B_1\mathcal{A}(a w b)B_2) &=\tr \left(B_1 \mathcal{A}(a)(\mathcal{A}(\tilde w) - \mathcal{A}(w))\mathcal{A}(b) B_2\right)\\
&\geq \mathfrak{d}(B_1)\mathfrak{d}(B_2) \tr \left( \mathcal{A}(a)(\mathcal{A}(\tilde w) - \mathcal{A}(w))\mathcal{A}(b) \right)\\
& \geq \mathfrak{d}(B_1)\mathfrak{d}(B_2)\end{align*}
as required. In the case where $\tilde b > a$ and $\tilde w > w$, the integer $k(w)$ and the matrix $A_0JA_1$ each contribute a negative sign to the product $\mathcal{A}(a)(\mathcal{A}(\tilde w) - \mathcal{A}(w))\mathcal{A}(b)) $ and the same conclusion may be reached.
\end{proof}
We may now prove the following two results which will allow us to characterise extremal orbits in terms of balanced words:

\begin{lemma}\label{finite-balanced-best}
Let $0 \leq \frac{p}{q} \leq 1$, with the integers $p$ and $q$ not necessarily coprime. Suppose that $|u|=q$, $|u|_1=p$ and
\begin{equation}\label{jt}\rho(\mathcal{A}(u)) = \max\left\{\rho(\mathcal{A}(v)) \colon |v|=q\text{ and }|v|_1=p\right\}.\end{equation}
Then the infinite word $u^\infty$ is balanced.
\end{lemma}
\begin{proof}
We shall begin by showing that if $u$ has the properties described then it is balanced. Let us assume for a contradiction that $u$ has these properties but is \emph{not} balanced. By Lemma~\ref{theys2}, there exists a suboptimal triple $(a,w,b)$ such that $a w b \prec u$. Let us write $u = s_1 a w b s_2$ and define $\hat u:=s_1 a \tilde w b s_2 $. By Lemma~\ref{sot2} we have $\tr(\mathcal{A}(\hat u)) > \tr(\mathcal{A}(u))$. Since $\mathcal{A}(\hat u)$ and $\mathcal{A}(u)$ are both non-negative matrices with unit determinant, it follows that
\begin{align*}
\rho(\mathcal{A}(\hat u)) &= \frac{1}{2}\left(\tr(\mathcal{A}(\hat u)) + \sqrt{\tr(\mathcal{A}(\hat u))^2  -4}\right)> \frac{1}{2}\left(\tr(\mathcal{A}(u)) + \sqrt{\tr(\mathcal{A}(u))^2  -4}\right) \\ &= \rho(\mathcal{A}(u)).
\end{align*}
Since clearly $|\hat u| = |u|$ and $|\hat u |_1 = |u|_1$ this is a contradiction, so $u$ must be balanced as required.

Now, suppose that $u$ satisfies \eqref{jt} with $|u|_1=p$ and $|u|=q$, and that $v$ is a cyclic permutation of $u$. It is a well-known property of the spectral radius that $\rho(B_1B_2)=\rho(B_2B_1)$ for any $B_1,B_2 \in \Mat$, and it follows from this that $\rho(\mathcal{A}(v))=\rho(\mathcal{A}(u))$. By applying the preceding argument to $v$ it follows that $v$ is also balanced.
We conclude that all of the cyclic permutations of $u$ are balanced, and by Lemma~\ref{cyclic-balanced} this implies that $u^\infty$ is balanced as required.
\end{proof}

\begin{proposition}\label{balanced1}
Let $\alpha \in (0,1]$ and suppose that $x \in Z_\alpha$. Then $x$ is balanced.
\end{proposition}
\begin{proof}
To prove the proposition, let us suppose that there exists a recurrent infinite word $x \in Z_\alpha$ which is not balanced. We shall then be able to deduce a contradiction, and the result follows. The general principle of the proof is that if $x$ is recurrent and not balanced, then we can construct a word based on $x$ along which the trace of the product $\mathcal{A}^{(\alpha)}(x,n)$ grows ``too rapidly''.

Fix a real number $C_\alpha>1$ such that $C_\alpha^{-1}\|B\|_\alpha \leq \|B\| \leq  C_\alpha \|B\|_\alpha$ for all $B\in \Mat$. By Lemma~\ref{gtr1} we have $\varrho(\alpha)>1$, and by Gelfand's formula we have $\left\|\mathcal{A}^{(\alpha)}\left(0^n\right)\right\|^{1/n}_\alpha \to 1$ as $n \to \infty$. It follows in particular that there is an integer $N_0 \geq 2$ such that $\left\|\mathcal{A}^{(\alpha)}\left(0^{N_0}\right)\right\|_\alpha < \varrho(\alpha)^{N_0}$ and therefore $0^{N_0} \nprec z$ for every $z \in Z_\alpha$. Similarly we may choose $N_1 \geq 2$ such that $1^{N_1} \nprec z$ for every $z \in Z_\alpha$. Let $N:= \max\{N_0,N_1\}$, and choose a further integer $M \geq 2$ such that
\[\max\left\{\left\|\mathcal{A}^{(\alpha)}\left(0^M\right)\right\|_\alpha, \left\|\mathcal{A}^{(\alpha)}\left(1^M\right)\right\|_\alpha\right\} <\frac{\varrho(\alpha)^M }{ 2C_\alpha N^2}.\]
If $v$ is any subword of $x$, then there exists $n \geq 0$ such that $\mathcal{A}^{(\alpha)}(v)=\mathcal{A}^{(\alpha)}(T^nx,|v|)$, and since $T^nx \in Z_\alpha$ this implies
\begin{equation}\label{pfbal1} \mathfrak{d}\left(\mathcal{A}^{(\alpha)}(v)\right)\geq \frac{1}{2N^2}\Lvvv\mathcal{A}^{(\alpha)}(v)\Rvvv \geq \frac{1}{2C_\alpha N^2}\left\|\mathcal{A}^{(\alpha)}\left(T^nx,|v|\right)\right\|_\alpha = \frac{\varrho(\alpha)^{|v|}}{2C_\alpha N^2},\end{equation}
where we have used Lemma~\ref{rho-norm}. On the other hand, for any nonempty finite word $u$,
\begin{equation}\label{pfbal4}\tr \mathcal{A}^{(\alpha)}(u) \leq 2\rho\left(\mathcal{A}^{(\alpha)}(u)\right) \leq 2 \left\|\mathcal{A}^{(\alpha)}(u)\right\|_\alpha \leq 2\varrho(\alpha)^{|u|}.\end{equation}
Now, since $x$ is not balanced, it by definition has a subword which is not balanced. Applying Lemma~\ref{sot2} to this subword we deduce that there exists a suboptimal triple $(a, w,b)$ such that $a w b \prec x$. Define $\ell:=|a w b|$, and fix an integer $K \geq 1$ such that
\[\left(1+\frac{\alpha^\ell}{16C_\alpha^2 N^4 M^2 \varrho(\alpha)^\ell}\right)^K> 2C_\alpha N^2.\]
Since $x$ is recurrent there are infinitely many occurrences of the word $a w b$ as a subword of $x$, and so we may choose words $s_1,\ldots,s_{K+1}$ such that the word
\[u^{(0)}:= s_1 (a w b) s_2 (a w b) s_3 \ldots\ldots s_K (a w b) s_{K+1}\]
is a subword of $x$. Let $L:=|u^{(0)}|$, and for $i=1,\ldots,K$ define a new word $u^{(i)}$ by reversing the first $i$ explicit instances of the word $w$ in $u^{(0)}$; that is,
\[u^{(1)}:= s_1 (a \tilde w b) s_2 (a w b) s_3 \ldots\ldots s_K (a w b) s_{K+1},\]
\[u^{(2)}:= s_1 (a \tilde w b) s_2 (a \tilde w b) s_3 \ldots\ldots s_K (a w b) s_{K+1},\]
and so forth, up to
\[u^{(K)}:= s_1 (a \tilde w b) s_2 (a \tilde w b) s_3 \ldots\ldots s_K (a \tilde w b) s_{K+1}.\]
Note that for each $i$ we have, by applying Lemma~\ref{sot2} $i$ times and using \eqref{pfbal1},
\begin{equation}\label{pfbal5}\tr \mathcal{A}^{(\alpha)}(u^{(i)}) \geq \tr \mathcal{A}^{(\alpha)}(u^{(0)}) \geq 2\mathfrak{d}\left(\mathcal{A}^{(\alpha)}(u^{(0)}) \right)\geq \frac{\varrho(\alpha)^L}{C_\alpha N^{2}},\end{equation}
since $u^{(0)}$ is a subword of $x$. As a consequence we observe that $0^M\nprec u^{(i)}$ for every $i$, since if we were to have $0^M \prec u^{(i)}$ for some $i$ then we could obtain
\begin{align*}\frac{\varrho(\alpha)^L}{2C_\alpha N^2} \leq \frac{1}{2}\tr \mathcal{A}^{(\alpha)}(u^{(i)}) &\leq \rho\left(\mathcal{A}^{(\alpha)}(u^{(i)})\right) \leq \left\|\mathcal{A}^{(\alpha)}(u^{(i)})\right\|_\alpha \\&\leq \left\|\mathcal{A}^{(\alpha)}\left(0^M\right)\right\|_\alpha . \varrho(\alpha)^{L-M} <\frac{\varrho(\alpha)^L}{2C_\alpha N^2},\end{align*}
a contradiction. Clearly an analogous contradiction would arise if we were to have $1^M \prec u^{(i)}$ and we conclude that $1^M \nprec u^{(i)}$ also.

Now, for $i=1,\ldots,K$ let $c^{(i)}$, $d^{(i)}$ be those words such that
$u^{(i-1)}=c^{(i)} a w b d^{(i)}$ and $u_{i}=c^{(i)} a \tilde w b d^{(i)}$. Note that $|c^{(i)}|+|d^{(i)}|+\ell = L$ for each $i$. Making $i$ applications of Lemma~\ref{sot2} and using \eqref{pfbal1} yields
\begin{align*}\tr \mathcal{A}^{(\alpha)}(c^{(i)}) &=\tr \mathcal{A}^{(\alpha)}\left(s_1(a\tilde w b)s_2 \ldots s_{i-1} (a\tilde w b)s_i\right) \\
&\geq \tr \mathcal{A}^{(\alpha)}\left(s_1(a w b)s_2 \ldots s_{i-1} (a w b)s_i\right) \geq \frac{\varrho(\alpha)^{|c^{(i)}|}}{2C_\alpha N^2},\end{align*}
since the last of these words is a subword of $u^{(0)}$, and $u^{(0)}$ is a subword of $x$. Since $c^{(i)} \prec u^{(i)}$ and $0^M, 1^M\nprec u^{(i)}$ we have $0^M , 1^M \nprec c^{(i)}$, and by Lemma~\ref{rho-norm} in combination with the preceding inequality this implies
\begin{equation}\label{pfbal2} \mathfrak{d}\left(\mathcal{A}^{(\alpha)}(c^{(i)})\right) \geq \frac{1}{4M^2}\tr \mathcal{A}^{(\alpha)}(c^{(i)}) \geq \frac{\varrho(\alpha)^{|c^{(i)}|}}{4C_\alpha N^2 M^2}.\end{equation}
Equally, since $d^{(i)} \prec u^{(0)}$ and $u^{(0)}$ is a subword of $x$, we may apply \eqref{pfbal1} to obtain
\begin{equation}\label{pfbal3}\mathfrak{d} \left(\mathcal{A}^{(\alpha)}(d^{(i)}) \right)\geq \frac{\varrho(\alpha)^{|d^{(i)}|}}{2C_\alpha N^2}.\end{equation}
We may now complete the proof. Combining \eqref{pfbal2}, \eqref{pfbal3}, and \eqref{pfbal4} we obtain for each $i$
\[\alpha^{|a w b|_1}\mathfrak{d}\left( \mathcal{A}^{(\alpha)}(c^{(i)})\right) \mathfrak{d}\left( \mathcal{A}^{(\alpha)}(d^{(i)})\right) \geq \frac{\alpha^\ell\varrho(\alpha)^{L-\ell}}{8C_\alpha^2 N^4 M^2}\geq \frac{\alpha^\ell}{16C_\alpha^2 N^4 M^2\varrho(\alpha)^\ell} \tr \mathcal{A}^{(\alpha)}(u^{(i-1)}),\]
and hence by Lemma~\ref{sot2},
\[\tr \mathcal{A}^{(\alpha)}(u^{(i)}) \geq \left(1+\frac{\alpha^\ell}{16C_\alpha^2 N^4M^2\varrho(\alpha)^\ell}\right)\tr \mathcal{A}^{(\alpha)}(u^{(i-1)}).\]
In combination with \eqref{pfbal4} and \eqref{pfbal5} this yields
\begin{align*}
2\varrho(\alpha)^L \geq \tr \mathcal{A}^{(\alpha)}(u^{(K)}) &\geq \left(1+\frac{\alpha^\ell}{16C_\alpha^2 N^4M^2\varrho(\alpha)^\ell}\right)^K\tr \mathcal{A}^{(\alpha)}(u^{(0)}) \\&\geq \left(1+\frac{\alpha^\ell}{16C_\alpha^2 N^4M^2\varrho(\alpha)^\ell}\right)^K \cdot\frac{\varrho(\alpha)^L}{C_\alpha N^2},
\end{align*}
contradicting our choice of $K$. The proof is complete.
\end{proof}

\section{Study of the growth of matrix products along balanced words}\label{section6}

In this section we analyse in detail the exponential growth rate of $\mathcal{A}(x,n)$ in the limit as $n \to \infty$ for $x \in X_\gamma$, investigating in particular the manner in which this value depends on $\gamma$. A construction with similar properties is discussed briefly in \cite[\S4.3]{BM}. The results of this section are summarised in the following proposition:

\begin{proposition}\label{Sproposition}
\begin{itemize}
\item There exists a continuous concave function $S \colon [0,1] \to \mathbb{R}$ such that for each $\gamma \in [0,1]$,
\[\lim_{n \to \infty} \frac{1}{n}\log\Lvvv\mathcal{A}(x,n)\Rvvv = \lim_{n \to \infty} \frac{1}{n}\log\rho(\mathcal{A}(x,n)) = S(\gamma)\]
uniformly for $x \in X_\gamma$.
\item If $\gamma=p/q \in [0,1]\cap\mathbb{Q}$ then $S(\gamma)=q^{-1}\log \rho(\mathcal{A}(x,q))$ for every $x \in X_\gamma$.
    \item The function $S$ also satisfies $\inf_{\gamma\in[0,1]} S = S(0)=S(1)=0$, $\sup S = S(1/2)=\log \varrho(1)$, and $S(\gamma)=S(1-\gamma)$ for all $\gamma \in [0,1]$.
\item The function $S$ is non-decreasing on $\left[0,\frac12\right]$.
    \end{itemize}
\end{proposition}

The proof of Proposition \ref{Sproposition} is given in the form of a sequence  of lemmas. Specifically, the result follows by combining Lemmas \ref{Sfunction}--\ref{Sbounds} and Lemma~\ref{Smain} below.

\begin{lemma}\label{Sfunction}
Let $\gamma \in [0,1]$. Then there exists a real number $S(\gamma)$ such that
\[\lim_{n \to \infty}\frac{1}{n}\log \Lvvv\mathcal{A}(x,n)\Rvvv=\lim_{n \to \infty}\frac{1}{n}\log\rho( \mathcal{A}(x,n))=S(\gamma)\]
uniformly over $x \in X_\gamma$.
\end{lemma}
\begin{proof}
In the cases $\gamma=0$, $\gamma=1$ the lemma is trivial, since by Theorem~\ref{Xphi} the set $X_\gamma$ consists of a single point which is fixed under $T$, and the result follows by Gelfand's formula. To prove the lemma in the nontrivial cases we use a result due to A. Furman \cite{F} on uniform convergence for linear cocycles over homeomorphisms. Since in general the transformations $T \colon X_\gamma \to X_\gamma$ are not homeomorphisms, this is achieved via an auxiliary construction.

Let us fix $\gamma \in (0,1)$. Define a space of two-sided sequences ${\hat X}_\gamma \subset \{0,1\}^{\mathbb{Z}}$ as follows: the sequence $x=(x_n)_{n \in \mathbb{Z}} \in \{0,1\}^{\mathbb{Z}}$ belongs to ${\hat{X}}_{\gamma}$ if and only if there exists $\delta \in [0,1]$ such that either $x_n \equiv  \lceil (n+1)\gamma + \delta \rceil - \lceil n\gamma + \delta \rceil$ for all $n \in \mathbb{Z}$, or $x_n \equiv  \lfloor (n+1)\gamma + \delta \rfloor - \lfloor n\gamma + \delta \rfloor$ for all $n \in \mathbb{Z}$. 

It follows from the discussion subsequent to the statement of Theorem \ref{Xphi} that the two-sided sequence $(x_i)_{i \in \mathbb{Z}}$ belongs to ${\hat{X}}_\gamma$ if and only if the one-sided sequence $(x_{i+k})_{i =1}^\infty$ belongs to $X_\gamma$ for every $k \in \mathbb{Z}$. We equip ${\hat{X}}_\gamma$ with the topology it inherits from the infinite product topology on $\{0,1\}^{\mathbb{Z}}$, and define $\hat{T} \colon {\hat{X}}_\gamma \to {\hat{X}}_\gamma$ by $\hat{T}[(x_i)_{i \in \mathbb{Z}}]:=(x_{i+1})_{i \in \mathbb{Z}}$ analogously to the definition of $T$. In the same manner as for the transformation $T \colon X_\gamma \to X_\gamma$, one may show that $\hat T \colon {\hat{X}}_\gamma \to {\hat{X}}_\gamma$ is a continuous, uniquely ergodic transformation of a compact metrisable space.

Finally, we define $\hat{\mathcal{A}} \colon {\hat{X}}_\gamma \times \mathbb{Z} \to \Mat$ in the following manner: given $x=(x_i)_{i \in \mathbb{Z}} \in {\hat{X}}_\gamma$ and $n \geq 1$, we define $\hat{\mathcal{A}}(x,n):=A_{x_n}\cdots A_{x_1}$, $\hat{\mathcal{A}}(x,-n):= A_{x_{-(n-1)}}^{-1}A_{x_{-(n-2)}}^{-1}\cdots A_{x_{0}}^{-1} = \hat{\mathcal{A}}(\hat{T}^{-n}x,n)^{-1}$, and $\hat{\mathcal{A}}(x,0)=I$. It may be directly verified that $\hat{\mathcal{A}}$ is continuous and satisfies the following cocycle relation: for all $x \in {\hat{X}}_\gamma$ and $n,m \in \mathbb{Z}$, we have $\hat{\mathcal{A}}(x,n+m)= \hat{\mathcal{A}}(\hat{T}^nx,m)\hat{\mathcal{A}}(x,n)$.

Now let $N\geq 1$ be as given by Lemma~ \ref{nolongwords}. For each $x \in X_\gamma$ we have $0^N, 1^N \nprec x$. Since for each $x = (x_i)_{i \in \mathbb{Z}} \in \hat{X}_\gamma$ we have $(x_i)_{i=1}^\infty \in X_\gamma$, it follows from this that the matrix product which defines $\hat{\mathcal{A}}(x,N)$ is a product of mixed powers of $A_0$ and $A_1$, and does not simply equal $A_0^N$ or $A_1^N$. A simple calculation shows that this implies that for each $x \in {\hat{X}}_\gamma$, all of the entries of the matrix $\hat{\mathcal{A}}(x,N)$ are strictly positive. We may therefore apply \cite[Theorem 3]{F} to deduce that there exists a real number $S(\gamma)$ such that $\frac{1}{n}\log \|\hat{\mathcal{A}}(x,n)\|$ converges uniformly to $S(\gamma)$ for $x \in {\hat{X}}_\gamma$. Since clearly for each $n \geq 1$,
\[\left\{\hat{\mathcal{A}}(x,n) \colon x \in {\hat{X}}_\gamma\right\} = \left\{\mathcal{A}(x,n)\colon x \in X_\gamma\right\},\]
this implies that $\frac{1}{n}\log \|\mathcal{A}(x,n)\|$ converges uniformly to $S(\gamma)$ for $x \in X_\gamma$. Since as previously noted we have $0^N, 1^N \nprec x$ for all $x \in X_\gamma$, it follows immediately from Lemma~\ref{rho-norm} that also $\frac{1}{n}\log \rho(\mathcal{A}(x,n)) \to S(\gamma)$ uniformly over $x \in X_\gamma$. The proof is complete.
\end{proof}

\begin{lemma}\label{Sestimates}
The function $S$ has the following properties:
\begin{enumerate}[(i)]
\item
Let $\gamma=p/q \in [0,1]$, not necessarily in least terms: then $S(\gamma)=q^{-1}\log \rho(\mathcal{A}(x,q))$ for every $x \in X_{p/q}$.
\item
Let $u$ be a finite word such that $|u|=q$, $|u|_1=p$. Then $S(p/q) \geq q^{-1}\log\rho(\mathcal{A}(u))$.
\item
Let $\gamma \in [0,1]$ be irrational. Then there exist $x \in X_\gamma$ and a sequence of rational numbers $(p_n/q_n)_{n=1}^\infty$ converging to $\gamma$ such that $S(p_n/q_n)=q_n^{-1}\log \rho(\mathcal{A}(x,q_n))$ for every $n \geq 1$.
\item
For every $\gamma \in [0,1]$ we have $S(\gamma)=S(1-\gamma)$.
\end{enumerate}
\end{lemma}
\begin{proof}
(i). By Theorem~\ref{Xphi} we have $T^qx=x$ for every $x \in X_{p/q}$, and so for every $x\in X_{p/q}$,
\[S(p/q) = \lim_{n \to \infty} \frac{1}{kq}\log\Lvvv\mathcal{A}(x,kq)\Rvvv = \lim_{k \to \infty}\frac{1}{kq}\log \Lvvv\mathcal{A}(x,q)^k\Rvvv = \frac{1}{q}\log \rho\left(\mathcal{A}(x,q)\right).\]

(ii). Clearly the set of all words $v$ such that $|v|=q$ and $|v|_1=p$ is finite, so there exists a word $v$ which attains the maximum value of $\rho(\mathcal{A}(v))$ within this set. In particular we have $\rho(\mathcal{A}(v))\geq \rho(\mathcal{A}(u))$. By Lemma~\ref{finite-balanced-best} the infinite word $v^\infty \in \Sigma$ is balanced, and since it is clearly recurrent we have $v^\infty \in X_{p/q}$ by Theorem~\ref{Xphi}. By part (i) this implies $q^{-1}\log \rho(\mathcal{A}(v))=S(p/q)$ as required.

(iii)
Let $x \in X_\gamma$ be as given by Lemma~\ref{standard}(ii), and let $(q_n)_{n=1}^\infty$ be a strictly increasing sequence of natural numbers such that $\pi_{q_n}(x)$ is a standard word for every $n$. Define $p_n:=|\pi_{q_n}(x)|_1$ for each $n \geq 1$. By the definition of $X_\gamma$ we have $p_n/q_n \to \gamma$. Since each $\pi_{q_n}(x)$ is standard, $[\pi_{q_n}(x)]^\infty \in X_{p_n/q_n}$ for each $n$ by Lemma~\ref{standard}(i), and by part (i) of the present lemma this implies $S(p_n/q_n)=q_n^{-1}\log \rho(\mathcal{A}(x,q_n))$.

(iv)
For each finite or infinite word $\omega$, define $\overline{\omega}$ to be the \emph{mirror image} of $\omega$, i.e., the unique word such that $\overline{\omega}_i=1$ if and only if $\omega_i=0$. It is clear that $x \in X_\gamma$ if and only if $\overline{x}\in X_{1-\gamma}$. Define $R=\left(\begin{smallmatrix}0&1\\1&0\end{smallmatrix}\right)$ and note that $R^{-1}A_0R=A_1$ and $R^{-1}A_1R=A_0$. If $x \in X_\gamma$ and $n \geq 1$, then
\[R^{-1}\mathcal{A}(x,n)R=(R^{-1}A_{x_n}R)\cdots (R^{-1}A_{x_2}R)(R^{-1}A_{x_1}R) = \mathcal{A}(\overline{x},n)\]
and in particular $\rho(\mathcal{A}(x,n))=\rho(\mathcal{A}(\overline{x},n))$. It follows easily that $S(\gamma)=S(1-\gamma)$.
\end{proof}

\begin{lemma}\label{Sbounds}
The function $S$ satisfies $S(0)=\inf S = 0$ and $S(\frac{1}{2})=\sup S = \log \varrho(1)$.
\end{lemma}
\begin{proof}
The reader may easily verify that
\begin{equation}\label{S1}\vvv A_1\vvv=\vvv A_0\vvv=\vvv A_0A_1\vvv^{\frac{1}{2}}=\vvv A_1A_0\vvv^{\frac{1}{2}}=\rho(A_0A_1)^{1/2}= \frac{1+\sqrt{5}}{2}.
\end{equation}
By Theorem~\ref{Xphi}, we have $X_{1/2}=\{(01)^\infty,(10)^\infty\}$, so by Gelfand's formula we have
\[\lim_{n \to \infty}\vvv\mathcal{A}(x,n)\vvv^{1/n} = \rho(A_0A_1)^{\frac{1}{2}} =  \rho(A_1A_0)^{\frac{1}{2}} = \frac{1+\sqrt{5}}{2}\]
when $x \in X_{1/2}$. Let us show that $\varrho(1)=\frac{1+\sqrt{5}}{2}$. In other words, we will prove that
\[
\sup\left\{\vvv\mathcal{A}(x,n)\vvv^{1/n}\colon x\in\Sigma\right\} = \lim_{n\to\infty}\vvv\mathcal{A}((01)^\infty,n)\vvv^{1/n} =\frac{1+\sqrt{5}}{2}.
\]
Suppose $x$ has a tail different from $(01)^\infty$. Then it must contain one of the following subwords: $w_1=11(01)^n1,\ w_2=11(01)^n00, w_3=00(10)^n0, w_4=00(10)^n11$ with $n\ge0$. In view of mirror symmetry, it suffices to deal with $w_1$ and $w_2$. We will show that it is possible to replace them with subwords of $(01)^\infty$,  $w_1'$ and $w_2'$ respectively, in such a way that the corresponding growth exponent does not decrease.

Namely, put $w_1'=(10)^{n+1}1$ and $w_2'=(10)^{n+2}$. It is easy to see that for $n\ge1$,
\begin{align*}
(A_0A_1)^n &= \begin{pmatrix} F_{2n} & F_{2n-1} \\ F_{2n-1} & F_{2n-2}  \end{pmatrix} \\
(A_1A_0)^n &= \begin{pmatrix} F_{2n-2} & F_{2n-1} \\ F_{2n-1} & F_{2n}  \end{pmatrix},
\end{align*}
where, as above, $(F_n)_{n=0}^\infty$ is the Fibonacci sequence (with $F_0=F_1=1$). Hence
\[
A_1^2(A_0A_1)^nA_1 = \begin{pmatrix} F_{2n+1} & F_{2n-1} \\ F_{2n+3} & F_{2n+1}  \end{pmatrix},
\]
whereas
\[
(A_1A_0)^{n+1}A_1 = \begin{pmatrix} F_{2n+2} & F_{2n+1} \\ F_{2n+3} & F_{2n+2}  \end{pmatrix},
\]
i.e., $\mathcal A(w_1')$ dominates $\mathcal A(w_1)$ entry-by-entry. Similarly,
\[
A_1^2(A_0A_1)^nA_0^2 = \begin{pmatrix} F_{2n} & F_{2n+2} \\ F_{2n+2} & F_{2n+4}  \end{pmatrix}
\]
and
\[
(A_1A_0)^{n+2} = \begin{pmatrix} F_{2n+2} & F_{2n+3} \\ F_{2n+3} & F_{2n+4}  \end{pmatrix}.
\]
Thus, $\varrho(1)= \frac{1+\sqrt{5}}{2}=e^{S(\frac{1}{2})}$, and since clearly $S(\gamma)\leq \log\varrho(1)$ for every $\gamma \in [0,1]$ this implies that $\sup S = S(1/2)$. On the other hand, it is clear that $X_0$ contains a single point $x$ corresponding to an infinite sequence of zeroes, and for this $x$ we have $S(0)=\log \rho(A_0)=0$. Finally, since every matrix $\mathcal{A}(x,n)$ is an integer matrix which has determinant one and is hence nonzero, every $x \in \Sigma$ has $\frac{1}{n}\log \vvv\mathcal{A}(x,n)\vvv \geq 0$ for all $n$ and therefore $S(\gamma)\geq 0$ for every $\gamma$.
\end{proof}

\begin{lemma}\label{Sconcave}
The restriction of $S$ to $(0,1)\cap\mathbb{Q}$ is concave in the following sense: if $\gamma_1,\gamma_2,\lambda \in (0,1) \cap \mathbb{Q}$ then $S(\lambda \gamma_1+(1-\lambda)\gamma_2) \geq \lambda S(\gamma_1)+(1-\lambda)S(\gamma_2)$.
\end{lemma}
\begin{proof}
For $i=1,2$ let $\gamma_i = p_i/q_i$ in least terms, and let $\lambda =k/m$. Let $M=\max\{q_1,q_2\}$. As a consequence of Lemma~\ref{Sestimates}(i) there exist finite words $u^{(1)},u^{(2)} \in \Omega$ such that $|u^{(i)}|_1=p_i$, $|u^{(i)}|=q_i$ and $S(\gamma_i) = q_i^{-1}\log \mathcal{A}(u^{(i)})$ for each $i$.

 Since $0<\gamma_1,\gamma_2<1$ we have $0<|p_i|<|q_i|$ and therefore $0^M, 1^M \nprec (u^{(i)})^\ell$ for $i=1,2$ and every $\ell \geq 1$. In particular, for each $\ell_1,\ell_2 \geq 1$ the word $(u^{(1)})^{\ell_1} (u^{(2)})^{\ell_2}$ does not have $0^{2M}$ or $1^{2M}$ as a subword, and hence by Lemma~\ref{rho-norm},
 \begin{align*}
\rho\bigl(\mathcal{A}(u^{(1)})^{\ell_1}(u^{(2)})^{\ell_2}\bigr) & \geq \mathfrak{d}\left(\mathcal{A}(u^{(1)})^{\ell_1}(u^{(2)})^{\ell_2}\right) \geq \mathfrak{d}(\mathcal{A}(u^{(1)})^{\ell_1})\mathfrak{d}(\mathcal{A} ((u^{(2)})^{\ell_2}) \\ &\geq \frac{1}{64M^4}\rho\left(\mathcal{A}(u^{(1)})^{\ell_1}\right) \rho\left(\mathcal{A}(u^{(2)})^{\ell_2}\right).
\end{align*}
Applying this inequality together with Lemma~\ref{Sestimates}(ii), for each $n \geq 1$ we obtain
\begin{align*}
S(\lambda \gamma_1 + (1-\lambda)\gamma_2) &= S\left(\frac{kp_1q_2 + (m-k)q_1p_2}{mq_1q_2}\right)\\
&\geq \frac{1}{nmq_1q_2}\log \rho(\mathcal{A}((u^{(1)})^{nkq_2}(u^{(2)})^{n(m-k)q_1}))\\
&\geq \frac{1}{nmq_1q_2}\left(\log \rho(\mathcal{A}((u^{(1)})^{nkq_2})) + \log \rho(\mathcal{A}((u^{(2)})^{n(m-k)q_1})) - \log 64M^4\right)\\
&= \frac{k}{mq_1}\log \rho(\mathcal{A}(u^{(1)})) + \frac{m-k}{mq_2}\log \rho(\mathcal{A}(u^{(2)})) - \frac{\log 64M^4}{nmq_1q_2}\\
&= \lambda S(\gamma_1) + (1-\lambda)S(\gamma_2) -\frac{\log 64M^4}{nmq_1q_2}.
\end{align*}
Taking the limit as $n \to \infty$ we obtain the desired result.
\end{proof}

\begin{lemma}\label{Smain}
The function $S \colon [0,1]\to \mathbb{R}$ is continuous and concave.
\end{lemma}
\begin{proof}
By Lemma~\ref{Sconcave}, the restriction of $S$ to $(0,1)\cap\mathbb{Q}$ is concave. Define a function $\widetilde S \colon [0,1] \to \mathbb{R}$ by
\[\widetilde S(\gamma) := \lim_{\varepsilon \to 0}\sup\left\{S(\gamma_*) \colon \gamma_* \in (0,1) \cap \mathbb{Q} \text{ and }|\gamma_*-\gamma|<\varepsilon\right\}.\]
Note that $\widetilde S$ is well-defined since $S$ is bounded by Lemma~ \ref{Sbounds}. We shall show in several stages that $\widetilde S$ is continuous, concave, and equal to $S$ throughout $[0,1]$. 

We first shall show that $\widetilde S$ is concave. Let $\gamma_1,\gamma_2,\lambda \in [0,1]$, and choose sequences of rationals $\left(\gamma_1^{(n)}\right)$, $\left(\gamma_2^{(n)}\right)$ and $\left(\lambda_n\right)$ belonging to $(0,1)$,  converging respectively to $\gamma_1, \gamma_2$ and $\lambda$, such that $\lim_{n \to \infty} S\left(\gamma_i^{(n)}\right) = \widetilde S(\gamma_i)$ for $i=1,2$. We then have
\begin{align*}\widetilde S\left(\lambda \gamma_1 + \left(1-\lambda\right)\gamma_2\right) &\geq \limsup_{n \to \infty} S\left(\lambda_n \gamma_1^{(n)} + \left(1-\lambda_n\right)\gamma_2^{(n)}\right)\\
&\geq \limsup_{n \to \infty} \lambda_n S\left(\gamma_1^{(n)}\right) + \left(1-\lambda_n\right)S\left(\gamma_2^{(n)}\right)\\
&= \lim_{n \to \infty}\lambda_n S\left(\gamma_1^{(n)}\right) + \left(1-\lambda_n\right)S\left(\gamma_2^{(n)}\right)\\& = \lambda \widetilde S(\gamma_1)+\left(1-\lambda\right)\widetilde S(\gamma_2)\end{align*}
using Lemma~\ref{Sconcave}, and $\widetilde S$ is concave as claimed. In particular the restriction of $\widetilde S$ to the interval $(0,1)$ is continuous (see for example \cite[Thm 10.3]{Rock}).

We next claim that $\widetilde S(\gamma)=S(\gamma)$ for rational values $0<\gamma<1$. Given $\gamma \in (0,1) \cap \mathbb{Q}$, choose a sequence of rationals $(\gamma_n)$ such that $\gamma_n \to \gamma$ and $S(\gamma_n) \to \widetilde S(\gamma)$. If $0<\gamma \leq \gamma_n$ for some $n$ then
\[S(\gamma) \geq  \left(1-\frac{\gamma}{\gamma_n}\right)S(0)+ \frac{\gamma}{\gamma_n}S(\gamma_n)= \frac{\gamma}{\gamma_n}S(\gamma_n),\] and similarly if $\gamma_n<\gamma < 1$ then
\[S(\gamma) \geq \left(\frac{1-\gamma}{1-\gamma_n}\right)S(\gamma_n) +\left(\frac{\gamma-\gamma_n}{1-\gamma_n}\right)S(1) \geq \left(\frac{1-\gamma}{1-\gamma_n}\right)S(\gamma_n).\]
It follows that by taking the limit as $n \to \infty$ we may obtain $S(\gamma) \geq \widetilde S(\gamma)$, and the converse inequality $\widetilde S(\gamma) \geq S(\gamma)$ is obvious from the definition of $\widetilde S$. This proves the claim.

We now claim that $\lim_{\gamma \to 0} \widetilde S(\gamma)=\widetilde S(0)=0=S(0)$ and $\lim_{\gamma \to 1}\widetilde S(\gamma)=\widetilde S(1)=0=S(1)$. Since $S(\gamma)=S(1-\gamma)$ for every $\gamma \in [0,1]$ by Lemma~\ref{Sestimates}(iv) it is sufficient to prove only the first assertion. By Lemma~\ref{Sbounds} we have $S(0)=\inf S =0 $ and therefore $\inf \widetilde S \geq 0$. Since $\widetilde S$ is concave there must exist $\delta>0$ such that the restriction of $\widetilde S$ to $[0,\delta)$ is monotone, and so if we can show that $\lim_{n \to \infty} \widetilde S(1/n)=0$ then the desired result will follow. By the preceding claim it is sufficient to show that $\lim_{n \to \infty} S(1/n)=0$. For each $n \geq 1$ it is easily verified using Lemma~\ref{cyclic-balanced} that $(0^n1)^\infty \in X_{1/n}$, so using Lemma~\ref{Sestimates}(i) we may estimate
\[0 \leq S\left(\frac{1}{n}\right)=\frac{1}{n+1}\log \rho(A_0^n A_1)  \leq \frac{1}{n+1} \log \tr (A_0^n A_1) = \frac{\log (n+2)}{n+1}\]
and therefore $S(1/n) \to 0$. This completes the proof of the claim.

To complete the proof of the lemma it suffices to show that in fact $\widetilde S(\gamma) = S(\gamma)$ when $\gamma$ is irrational. Given $\gamma \in [0,1] \setminus \mathbb{Q}$, let $x \in X_\gamma$ and $(p_n/q_n)_{n=1}^\infty$ be as given by Lemma~\ref{Sestimates}(iii). Since $\widetilde S$ is continuous and agrees with $S$ on the rationals, we may apply parts (iii) and (i) of Lemma~\ref{Sestimates} to obtain
\[S(\gamma)= \lim_{n \to \infty} \frac{1}{q_n}\log \rho(\mathcal{A}(x,q_n)) = \lim_{n \to \infty} S\left(\frac{p_n}{q_n}\right) = \lim_{n \to \infty}\widetilde S\left(\frac{p_n}{q_n}\right) = \widetilde S(\gamma),\]
and we conclude that $\widetilde S \equiv S$ as desired.
\end{proof}

To conclude the proof of Proposition~\ref{Sproposition}, we note that the function $S$ being non-decreasing on $\left[0,\frac12\right]$ follows from its concavity and the fact that $\max\limits_{\gamma\in[0,1/2]} S(\gamma)=S(1/2)$.

\section{Proof of Theorem~\ref{technical}}
\label{sec7}

Before commencing the proof of Theorem~\ref{technical}, we require the following simple lemma:
\begin{lemma}\label{techpf}
For each $\alpha \in [0,1]$ we have $\varrho(\alpha) \geq e^{S(\gamma)}\alpha^\gamma$ for all $\gamma \in [0,1]$. If $\alpha \in (0,1]$ and $X_\gamma \cap Z_\alpha \neq \emptyset$, then $X_\gamma \subseteq Z_\alpha$ and $\varrho(\alpha)=e^{S(\gamma)}\alpha^\gamma$.
\end{lemma}
\begin{proof}
In the case $\alpha=0$, an easy calculation using Proposition \ref{Sproposition} and the definition of $\varrho$ shows that $\varrho(\alpha)=\rho(A_0)=1=e^{S(0)}$. It is therefore clear in this case that $\varrho(\alpha)=e^{S(\gamma)}\alpha^\gamma$ if and only if $\gamma=0$. For the rest of the proof let us fix $\alpha \in (0,1]$ and $\gamma \in [0,1]$. For each $x \in X_\gamma$, we have
\begin{align*}\log \varrho(\alpha) &=\limsup_{n \to \infty} \sup\left\{\frac{1}{n}\log \Lvvv\mathcal{A}^{(\alpha)}(z,n)\Rvvv \colon z \in \Sigma\right\}\geq \lim_{n \to \infty} \frac{1}{n}\log \Lvvv\mathcal{A}^{(\alpha)}(x,n)\Rvvv \\&= \lim_{n \to \infty} \left(\frac{1}{n}\log \vvv\mathcal{A}(x,n)\vvv + \varsigma(\pi_n(x))\log \alpha\right) = S(\gamma)+\gamma \log \alpha\end{align*}
so that $\varrho(\alpha) \geq e^{S(\gamma)}\alpha^\gamma$. If $x \in X_\gamma \cap Z_\alpha$ then by the definition of $Z_\alpha$ we have
\[S(\gamma)+\gamma\log\alpha=\lim_{n \to \infty} \frac{1}{n}\log \Lvvv\mathcal{A}^{(\alpha)}(x,n)\Rvvv = \lim_{n \to \infty} \frac{1}{n}\log \left\|\mathcal{A}^{(\alpha)}(x,n)\right\|_\alpha = \log \varrho(\alpha)\]
so that $\varrho(\alpha)=e^{S(\gamma)}\alpha^\gamma$, and since by Theorem~\ref{Xphi} the restriction of $T$ to $X_\gamma$ is minimal it is clear that $X_\gamma \subseteq Z_\alpha$.
\end{proof}
We also require the following lemma, which is an easy consequence of a result in \cite{BTV}:
\begin{lemma}\label{varsi}
Let $\alpha \in [0,1]$ and let $u,v$ be nonempty finite words such that $\rho(\mathcal{A}^{(\alpha)}(u))^{1/|u|} =\rho(\mathcal{A}^{(\alpha)}(v))^{1/|v|} =\varrho(\alpha)$. Then $\varsigma(u)=\varsigma(v)$.
\end{lemma}
\begin{proof}
In \cite{BTV}, Blondel, Theys and Vladimirov define two nonempty finite words $u,v$ to be \emph{essentially equal} if there exist finite words $a,b$ such that $au^\infty = bv^\infty$. In particular it is clear that if $u$ and $v$ are essentially equal then necessarily $\varsigma(u)=\varsigma(v)$. Blondel \emph{et al.} then associate to each nonempty finite word $\omega$ the set $J_\omega=\{\alpha \in [0,1] \colon \mathcal{A}^{(\alpha)}(\omega) = \varrho(\alpha)^{|\omega|}\}$. In \cite[Lemma ~4.4]{BTV} it is shown that if $J_u \cap J_v \neq \emptyset$ then $u$ and $v$ are essentially equal. We deduce from this that if $u$ and $v$ are nonempty finite words which satisfy $\rho(\mathcal{A}^{(\alpha)}(u))^{1/|u|} =\rho(\mathcal{A}^{(\alpha)}(v))^{1/|v|} =\varrho(\alpha)$ for some fixed $\alpha \in [0,1]$, then $\alpha \in J_u \cap J_v$ by definition; this implies that $u$ and $v$ are essentially equal, and therefore $\varsigma(u)=\varsigma(v)$.
\end{proof}
Now we are ready to prove Theorem~\ref{technical}.

\medskip\noindent \textbf{1. Existence of $\mathfrak r$.}
We shall begin by showing that for each $\alpha \in (0,1]$ there exists a unique $\gamma \in [0,1]$ such that $X_{\gamma} \cap Z_\alpha \neq \emptyset$. Let $\alpha \in (0,1]$. By Lemma~\ref{Zset} the set $Z_\alpha$ is compact and invariant under $T$, and this implies that it contains a recurrent point (see e.g. \cite[p.130]{KH}). It follows by Proposition~\ref{balanced1} that $Z_\alpha$ contains a recurrent balanced infinite word, and hence there exists $\gamma_\alpha \in [0,1]$ such that $X_{\gamma_\alpha} \cap Z_\alpha \neq \emptyset$. By Lemma~ \ref{techpf} it follows that $e^{S(\gamma_\alpha)}\alpha^{\gamma_\alpha}=\varrho(\alpha)$. We claim that $\gamma_\alpha$ is the unique element of $ [0,1]$ with this property. By Lemma~\ref{techpf} this further implies that $X_\gamma \cap Z_\alpha = \emptyset$ when $\gamma\neq\gamma_\alpha$.

To prove this claim, let us suppose that $0 \leq \gamma_1 < \gamma_2 \leq 1$ with $e^{S(\gamma_1)}\alpha^{\gamma_1} =e^{S(\gamma_2)}\alpha^{\gamma_2} =\varrho(\alpha)$, and derive a contradiction. Choose $\lambda_1,\lambda_2 \in [0,1]$ such that $\tilde\gamma_1:=\lambda_1 \gamma_1 + (1-\lambda_1)\gamma_2$ and $\tilde\gamma_2:=\lambda_2 \gamma_1 + (1-\lambda_2)\gamma_2$ are both rational with $\gamma_1\leq \tilde \gamma_1 <\tilde \gamma_2 \leq \gamma_2$. Applying Proposition~\ref{Sproposition} we deduce
\begin{align*}
S(\tilde\gamma_i)+\tilde\gamma_i\log\alpha &= S(\lambda_i\gamma_1+(1-\lambda_i)\gamma_2)+ (\lambda_i\gamma_1+(1-\lambda_i)\gamma_2)\log\alpha\\
&\geq \lambda_i(S(\gamma_1)+\gamma_1\log\alpha) +(1-\lambda_i)(S(\gamma_2)+\gamma_2\log\alpha) = \log\varrho(\alpha),
\end{align*}
and hence $e^{S(\tilde\gamma_i)}\alpha^{\tilde\gamma_i} \geq \varrho(\alpha)$, for $i=1,2$. Applying Lemma~\ref{techpf} it follows that $e^{S(\tilde\gamma_1)}\alpha^{\tilde\gamma_1} =e^{S(\tilde\gamma_2)}\alpha^{\tilde\gamma_2}=\varrho(\alpha)$. Let $x \in X_{\tilde\gamma_1}$ and $y \in X_{\tilde\gamma_2}$, and let $u:=\pi_{q_1}(x)$ and $v:=\pi_{q_2}(y)$. By Proposition~\ref{Sproposition} we have $\varrho(\alpha)=\rho(\mathcal{A}^{\alpha}(u))^{1/|u|} = \rho(\mathcal{A}^{\alpha}(v))^{1/|v|}$, and since $\varsigma(u)=\tilde\gamma_1<\tilde\gamma_2=\varsigma(v)$ this contradicts Lemma~\ref{varsi}. The claim is proved.

Let us define $\mathfrak{r}(\alpha):=\gamma_\alpha$ for all $\alpha \in (0,1]$, and $\mathfrak{r}(0):=0$. Note that $\varrho(0)=\rho(A_0)=1=e^{S(0)}$ as a consequence of Lemma~\ref{bowf} and Proposition \ref{Sproposition}. It follows from this and the previous arguments that for all $\alpha,\gamma \in [0,1]$ we have $\varrho(\alpha) \geq e^{S(\gamma)}\alpha^\gamma$ with equality if and only if $\gamma=\mathfrak{r}(\alpha)$, and for all $\alpha \in (0,1]$ we have $X_\gamma \cap Z_\alpha \neq \emptyset$ precisely when $\gamma = \mathfrak{r}(\alpha)$, in which case $X_{\mathfrak{r}(\alpha)}\subseteq Z_\alpha$.

\medskip\noindent \textbf{2. Monotonicity of $\mathfrak r$.}
We now show that the function $\mathfrak{r}$ thus defined is non-decreasing.  Let us suppose that $\alpha_1,\alpha_2 \in [0,1]$ with $\mathfrak{r}(\alpha_1)<\mathfrak{r}(\alpha_2)$; this implies in particular that $\alpha_2$ is nonzero. By the preceding result we have $\varrho(\alpha_1)= e^{S(\mathfrak{r}(\alpha_1))}\alpha_1^{\mathfrak{r}(\alpha_1)}> e^{S(\mathfrak{r}(\alpha_2))}\alpha_1^{\mathfrak{r}(\alpha_2)}$ and similarly $\varrho(\alpha_2)= e^{S(\mathfrak{r}(\alpha_2))}\alpha_2^{\mathfrak{r}(\alpha_2)}> e^{S(\mathfrak{r}(\alpha_1))}\alpha_2^{\mathfrak{r}(\alpha_1)}$. Consequently $\alpha_1^{\mathfrak{r}(\alpha_2)-\mathfrak{r}(\alpha_1)} <e^{S(\mathfrak{r}(\alpha_1))-S(\mathfrak{r}(\alpha_2))} < \alpha_2^{\mathfrak{r}(\alpha_2)-\mathfrak{r}(\alpha_1)}$, and since $\mathfrak{r}(\alpha_2)-\mathfrak{r}(\alpha_1)>0$ we deduce that $\alpha_1<\alpha_2$. We conclude that if $0\leq \alpha_1<\alpha_2\leq 1$ then necessarily $\mathfrak{r}(\alpha_1)\leq \mathfrak{r}(\alpha_2)$ and therefore $\mathfrak{r}$ is non-decreasing as required.

\medskip\noindent \textbf{3. Continuity of $\mathfrak r$.}
We may now show that $\mathfrak{r}$ is continuous. Given $\alpha_0 \in (0,1]$ let $\mathfrak{r}_-$ be the limit of $\mathfrak{r}(\alpha)$ as $\alpha \to \alpha_0$ from the left, which exists since $\mathfrak{r}$ is monotone. For every $\alpha \in (0,1]$ we have $\varrho(\alpha)= e^{S(\mathfrak{r}(\alpha))}\alpha^{\mathfrak{r}(\alpha)}$. By Lemma~\ref{rhocts} and Proposition \ref{Sproposition}, $\varrho$ and $S$ are continuous, so taking the left limit at $\alpha_0$ yields $e^{S(\mathfrak{r}(\alpha_0))}\alpha_0^{\mathfrak{r}(\alpha_0)}= \varrho(\alpha_0)=e^{S(\mathfrak{r}_-)}\alpha_0^{\mathfrak{r}_-}$. Since $\mathfrak{r}(\alpha)$ is the unique value for which this equality may hold we deduce that $\mathfrak{r}(\alpha_0)= \mathfrak{r}_-$ as required. Similarly for every $\alpha_0 \in [0,1)$ the limit of $\mathfrak{r}(\alpha)$ as $\alpha \to\alpha_0$ from the right is equal to $\mathfrak{r}(\alpha_0)$, and we conclude that $\mathfrak{r}$ is continuous. Since $\mathfrak{r}(0)=0$ and $\mathfrak{r}(1)=1/2$ as a consequence of Proposition \ref{Sproposition}, and we have shown that $\mathfrak{r}$ is continuous and monotone, we deduce that $\mathfrak{r}$ maps $[0,1]$ surjectively onto $[0,\frac{1}{2}]$ as claimed.

\medskip\noindent \textbf{4. 1-ratio and characterisation of extremal orbits.}
It remains to show that for each $\alpha$ the extremal orbits of $\mathsf{A}_\alpha$ may be characterised in terms of $X_{\mathfrak{r}(\alpha)}$ in the manner described by the Theorem, and that $\mathfrak{r}(\alpha)$ is the unique optimal $1$-ratio of $\mathsf{A}_\alpha$. In the case $\alpha=0$ it is obvious that $x \in \Sigma$ is weakly extremal if and only if it is strongly extremal, if and only if $x =0^\infty \in X_0$, and in this case the proof is then complete. For each $\alpha \in (0,1]$, Lemma~\ref{strongex} shows that every recurrent strongly extremal infinite word belongs to $Z_\alpha$, and therefore belongs to $X_{\mathfrak{r}(\alpha)}$ by Proposition \ref{balanced1} and the uniqueness property of $\mathfrak{r}(\alpha)$.

To show that weakly extremal infinite words accumulate on $X_{\mathfrak{r}(\alpha)}$ in the desired manner we require an additional claim. Given $\alpha \in (0,1]$, we assert that there is a unique $T$-invariant Borel probability measure whose support is contained in $Z_\alpha$, and that this support is equal to $X_{\mathfrak{r}(\alpha)}$. Indeed, let $\mu_{\mathfrak{r}(\alpha)}$ be the unique $T$-invariant measure with support equal to $X_{\mathfrak{r}(\alpha)}$, the existence of which is given by Theorem~\ref{Xphi}. If $\nu$ is a $T$-invariant Borel probability measure with $\supp \nu \subseteq Z_\alpha$, define $\widetilde X := \{x \in \supp \nu \colon x \text{ is recurrent}\}$. It follows from the Poincar\'e recurrence theorem that $\widetilde X$ is dense in $\supp \nu$ (see e.g. \cite[Prop. 4.1.18]{KH}). By Proposition \ref{balanced1} every element of $\widetilde X$ is balanced, and since $\mathfrak{r}(\alpha)$ is the unique $\gamma \in [0,1]$ for which $X_\gamma \cap Z_\alpha \neq \emptyset$, it follows that $\widetilde X \subseteq X_{\mathfrak{r}(\alpha)}$. We conclude that $\supp \nu \subseteq X_{\mathfrak{r}(\alpha)}$ and therefore $\nu=\mu_{\mathfrak{r}(\alpha)}$ since the restriction of $T$ to $X_{\mathfrak{r}(\alpha)}$ is known to be uniquely ergodic, which proves the claim. By Theorem~\ref{Xphi} we have $\mu_{\mathfrak{r}(\alpha)}(\{x \in \Sigma \colon x_1=1\}) =\mathfrak{r}(\alpha)$, and we may now apply Lemma~ \ref{weakex} to see that if $x \in \Sigma$ is weakly extremal for $\mathsf{A}_\alpha$, then $(1/n) \sum_{k=0}^{n-1} \mathrm{dist}(x,X_{\mathfrak{r}(\alpha)}) \to 0$ and $\varsigma(\pi_n(x)) \to \mathfrak{r}(\alpha)$ as required.

It remains only to show that for each $\alpha \in (0,1]$, every $x \in X_{\mathfrak{r}(\alpha)}$ is strongly extremal in the strict fashion described by \eqref{niceformula}. Given any compact set $K \subseteq (0,1]$, choose an integer $N_K$ large enough that $N_K>\max\{\lceil \mathfrak{r}(\alpha)^{-1}\rceil, \lceil (1-\mathfrak{r}(\alpha))^{-1}\rceil\}$ for every $\alpha \in K$, and let $M_K>1$ be the constant given by Lemma~\ref{extnorm}. Let $\alpha \in K$ and $x \in X_{\mathfrak{r}(\alpha)}$. By Lemma~\ref{nolongwords} we have $0^{N_K}, 1^{N_K} \nprec x$, and since $X \subseteq Z_\alpha$ we have $\|\mathcal{A}^{(\alpha)}(x,n)\|_\alpha=\varrho(\alpha)^n$ for all $n \geq 1$. Applying Lemma~\ref{rho-norm} and Lemma~ \ref{extnorm},
\begin{align*}\frac{\varrho(\alpha)^n}{2M_KN_K^2} &= \frac{1}{2M_KN_K^2}\left\|\mathcal{A}^{(\alpha)}(x,n) \right\|_\alpha \leq \frac{1}{2N_K^2}\Lvvv\mathcal{A}^{(\alpha)}(x,n)\Rvvv \leq \rho\left(\mathcal{A}^{(\alpha)}(x,n)\right)\\& \leq \Lvvv\mathcal{A}^{(\alpha)}(x,n)\Rvvv \leq M_K\left\|\mathcal{A}^{(\alpha)}(x,n)\right\|_\alpha \leq M_K \varrho(\alpha)^n<2M_KN_K^2\varrho(\alpha)^n\end{align*}
so that \eqref{niceformula} holds with $C_K:=2M_K N^2_K$. In particular this shows that for each $\alpha \in (0,1]$, every $x \in X_{\mathfrak{r}(\alpha)}$ is strongly extremal. The proof of the Theorem is complete.

\section{Proof of Theorem~\ref{counter}}
\label{sec8}

Recall from Proposition~\ref{Sproposition} that there exists a continuous concave function $S \colon [0,1] \to \mathbb{R}$ such that for each $\gamma \in [0,1]$,
\[S(\gamma)= \lim_{n \to \infty} \frac{1}{n} \log \vvv\mathcal{A}(x,n)\vvv =\lim_{n \to \infty} \frac{1}{n} \log \rho(\mathcal{A}(x,n))\]
uniformly for $x \in X_\gamma$. We saw in the course of the proof of Theorem~\ref{technical} that the function $\mathfrak{r} \colon [0,1] \to [0,\frac{1}{2}]$ is characterised by the fact that $\varrho(\alpha) \geq e^{S(\gamma)}\alpha^\gamma$ for all $\alpha,\gamma \in [0,1]$ with equality if and only if $\gamma=\mathfrak{r}(\alpha)$. Readers who have skipped the proof of Theorem~\ref{technical} may note that this characterisation can be deduced easily from the definition of $S$ and the statement of Theorem~\ref{technical}.

The proof of Theorem~\ref{counter} operates by exploiting the concavity of $S$ and the above relationship between $S$ and $\mathfrak{r}$ to compute a value $\alpha_* \in [0,1]$ such that $\mathfrak{r}(\alpha_*) \notin \mathbb{Q}$ as the limit of a series of approximations. We begin with a result  from convex analysis.
\begin{lemma}\label{deriv}
For each $\gamma \in \left(0,\frac{1}{2}\right)$, we have $\mathfrak{r}^{-1}(\gamma)=\{\alpha_0\}$ with $\alpha_0 \in (0,1]$ if and only if $S$ is differentiable at $\gamma$ and $S'(\gamma)=-\log \alpha_0$.
\end{lemma}

See Figure~\ref{fig:S(gamma)} for a graph of $S(\gamma)$ along with the tangent line of slope~$\alpha_*$.

\begin{figure}[H]
\[ \includegraphics[width=250pt,height=250pt]{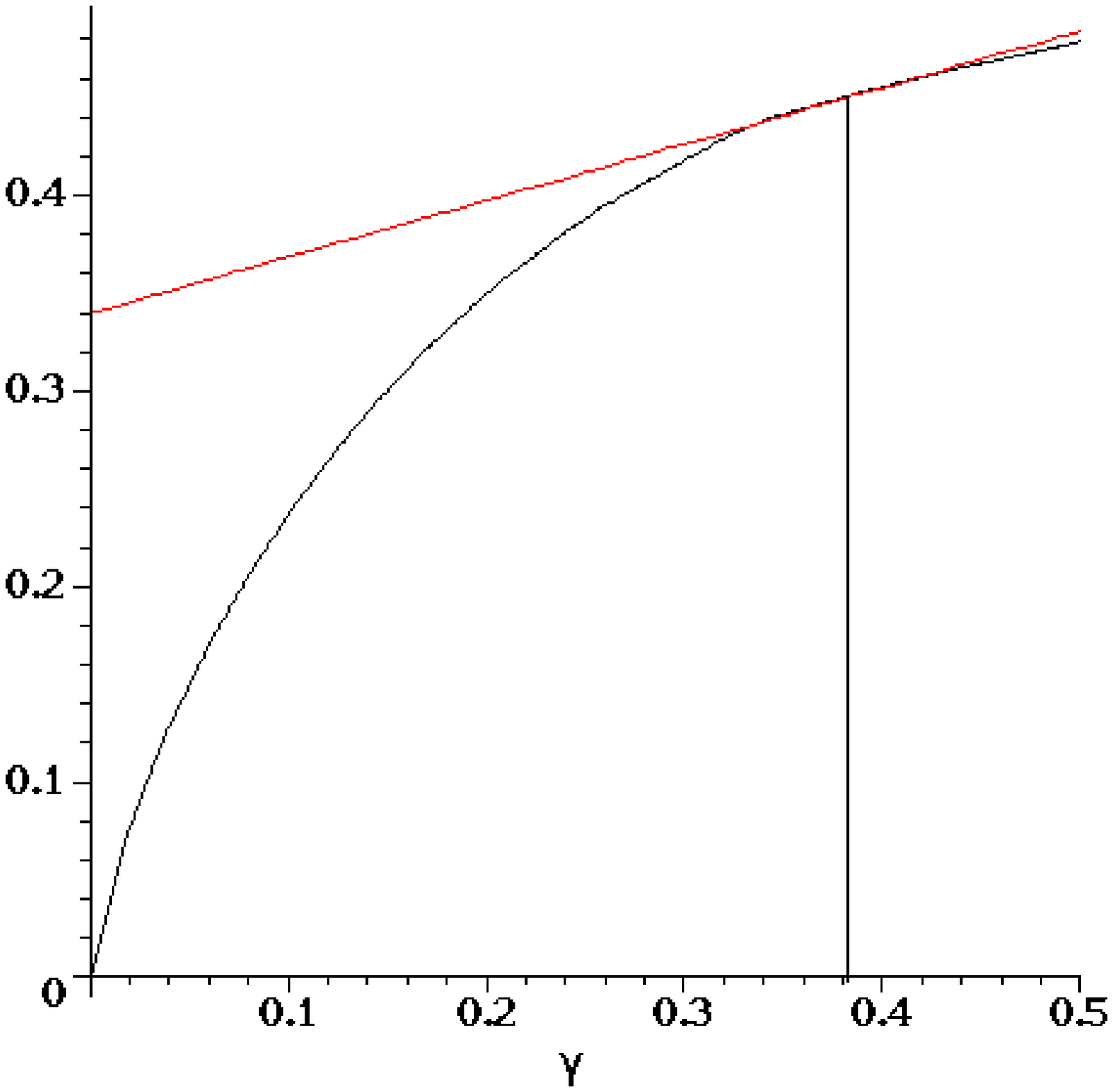} \]
\caption{Graph of $S(\gamma)$, and tangent line at $\gamma \approx 0.3819660\dots$ of slope $-\log(0.74932\dots)$}
\label{fig:S(gamma)}
\end{figure}

\begin{proof}
Recall that if $f \colon [a,b] \to \mathbb{R}$ is a concave function then $\eta \in \mathbb{R}$ is called a \emph{subgradient} of $f$ at $z \in [a,b]$ if $f(y) \leq f(z) + \eta(y-z)$ for all $y \in [a,b]$. Furthermore, $f$ is differentiable at $z \in (a,b)$ with $f'(z)=\eta$ if and only if $\eta$ is the unique subgradient of $f$ at $z$ (see for example \cite[Thm 25.1]{Rock}).
To prove the lemma it therefore suffices to show that for each $\gamma \in (0,\frac{1}{2})$, $\eta\in \mathbb{R}$ is a subgradient of $S$ at $\gamma$ if and only if $e^{-\eta} \in (0,1]$ and $\mathfrak{r}(e^{-\eta})=\gamma$.

Let us prove that this is the case. For every $\alpha,\gamma\in [0,1]$ we have $e^{S(\mathfrak{r}(\alpha))}\alpha^{\mathfrak{r}(\alpha)} \geq e^{S(\gamma)}\alpha^\gamma$ with equality if and only if $\gamma=\mathfrak{r}(\alpha)$. For each fixed $\alpha \in (0,1)$ it now follows by a simple rearrangement that $-\log \alpha$ is a subgradient of $S$ at $\mathfrak{r}(\alpha)$. Conversely, suppose that $\eta \in \mathbb{R}$ is a subgradient of $S$ at some $\gamma_0 \in (0,\frac{1}{2})$. By Proposition~\ref{Sproposition}, $S$ is monotone increasing on the interval $[0,\frac{1}{2}]$ and therefore we must have $\eta \geq 0$. Since $\eta$ is a subgradient we have $e^{S(\gamma_0)-\eta \gamma_0} \geq e^{S( \gamma)-\eta \gamma}$ for all $\gamma \in [0,1]$, and since $e^{-\eta} \in (0,1]$ it follows that $\gamma_0=\mathfrak{r}(e^{-\eta})$ as required.
\end{proof}

The following corollary is not needed in this paper but since it is straightforward, we believe it's worth mentioning.

\begin{corollary}The function $S$ is strictly concave on $[0,1]$ and strictly increasing on $[0,1/2]$.
\end{corollary}
\begin{proof}If $S$ were not strictly concave, there would be an interval $(\gamma_1,\gamma_2)$ such that $S$ would be linear on this interval. Hence $S'$ would be constant on $(\gamma_1,\gamma_2)$ which would mean, in view of the previous lemma, that $\mathfrak r^{-1}(\gamma)$ would be constant for all $\gamma\in (\gamma_1,\gamma_2)$. This contradicts $\mathfrak r$ being well defined (Theorem~\ref{technical}), whence $S$ is strictly concave.

Since $S$ is non-decreasing, continuous and strictly concave on $[0,1/2]$, it is strictly increasing.
\end{proof}

Throughout this section we let $\phi:=\frac{1+\sqrt{5}}{2}$ denote the golden ratio. Recall that a real number $\gamma$ is said to be \emph{Liouville} if for every $k>0$ there exist integers $p,q$ such that $0<|\gamma-p/q|< 1/q^k$. A classical theorem of Liouville asserts that no algebraic number can be Liouville (see, e.g., \cite[Theorem~191]{HW}). In particular $\phi^{-2}$ is not Liouville.
\begin{lemma}\label{FIM}
Let $\gamma \in [0,\frac{1}{2}]$ and suppose that $\gamma$ is an irrational number which is not Liouville. Then there exists a unique $\alpha \in [0,1]$ such that $\mathfrak{r}(\alpha)=\gamma$.
\end{lemma}
\begin{proof}
By Theorem~\ref{technical} the function $\mathfrak{r}$ is surjective and monotone, so the set $\mathfrak{r}^{-1}(\gamma)$ is either a point or an interval. To show that this set cannot be an interval, we shall suppose that there exist $\alpha_0 \in (0,1)$ and $\varepsilon>0$ such that $\mathfrak{r}(\alpha)=\gamma$ for all $\alpha \in [e^{-\varepsilon}\alpha_0, e^{\varepsilon}\alpha_0]$, and derive a contradiction.

Since $\gamma$ is irrational but not Liouville, we may choose an integer $k>0$ such that for all integers $p,q$ with $q$ nonzero we have $\left|\gamma - p/q\right| > 1/q^k$. A theorem due to the second named author \cite[Thm 1.2]{QBWF} implies that for every $r>0$,
\[\max\left\{\rho\left(\mathcal{A}^{(\alpha_0)}(x,m)\right)^{\frac{1}{m}} \colon x \in \Sigma \text{ and }1 \leq m \leq n\right\} = \varrho(\alpha_0)+O\left(\frac{1}{n^r}\right)\]
in the limit as $n \to \infty$. In particular it follows that if $n$ is some sufficiently large integer, then there exist an integer $m$ and an infinite word $x \in \Sigma$ such that $1 \leq m \leq n$ and
\begin{equation}\label{rpoint2}\rho\left(\mathcal{A}^{(\alpha_0)}(x,m)\right)^{1/m} >\left(1-\frac{1}{n^{k+1}}\right)\varrho(\alpha)>e^{-\varepsilon n^{-k}}\varrho(\alpha).\end{equation}
Let $\varsigma(\pi_m(x))=p/q$ in least terms; we shall suppose firstly that $\frac{p}{q}-\gamma>0$, the opposite case being similar. By hypothesis we have $\varrho(\lambda\alpha_0)= e^{S(\mathfrak{r}(\lambda\alpha_0))}(\lambda\alpha_0)^{\mathfrak{r}(\alpha_0)} = e^{S(\gamma)}(\lambda\alpha_0)^\gamma=\lambda^{\gamma}\varrho(\alpha_0)$ for every $\lambda \in [e^{-\varepsilon},e^{\varepsilon}]$, and also $\frac{p}{q}-\gamma = |\frac{p}{q}-\gamma| > q^{-k} \geq n^{-k}$. Combining this with \eqref{rpoint2} and Lemma~ \ref{bowf} we obtain
\begin{align*}\varrho\left(e^{\varepsilon}\alpha_0\right) \geq \rho\left(\mathcal{A}^{\left(e^{\varepsilon}\alpha_0\right)}(x,m)\right)^{1/m} &=e^{\varepsilon p/q} \rho\left(\mathcal{A}^{(\alpha_0)}(x,m)\right)^{1/m}\\&> e^{\varepsilon p/q-\varepsilon n^{-k}}\varrho(\alpha_0) = e^{\varepsilon\left(p/q - \gamma- n^{-k}\right)}\varrho\left(e^{\varepsilon}\alpha_0\right)\\&> \varrho\left(e^{\varepsilon}\alpha_0\right),\end{align*}
a contradiction. In the case $\frac{p}{q}-\gamma<0$ we may similarly arrive at the expression
\[\varrho\left(e^{-\varepsilon}\alpha_0\right) >e^{-\varepsilon p/q - \varepsilon n^{-k}}\varrho(\alpha_0)= e^{\varepsilon\left(\gamma-p/q-n^{-k}\right)} \varrho\left(e^{-\varepsilon}\alpha_0\right) >\varrho\left(e^{-\varepsilon}\alpha_0\right)\]
which is also a contradiction. The proof is complete.
\end{proof}

Let $(F_n)_{n=0}^\infty$ denote the Fibonacci sequence, which is defined by $F_0:=0$, $F_1:=1$ together with the recurrence relation $F_{n+2}:=F_{n+1}+F_n$, and recall that $F_n = (\phi^n - (-1/\phi)^n)/\sqrt{5}$ for every $n \geq 0$. Define a sequence of integers $(\tau_n)_{n=0}^\infty$ by $\tau_0:=1$, $\tau_1=\tau_2:=2$, and $\tau_{n+1}:=\tau_{n}\tau_{n-1}-\tau_{n-2}$ for every $n \geq 2$. Finally, define a sequence of matrices $(B_n)_{n=1}^\infty$ by $B_1:=A_1$, $B_2:=A_0$ and $B_{n+1}:=B_{n}B_{n-1}$ for every $n \geq 2$. The key properties of $F_n$, $B_n$ and $\tau_n$ are summarised in the following three lemmas.
\begin{lemma}\label{bn-matrices}
For each $n \geq 2$ the identities $S(F_{n-2}/F_n)=F_n^{-1}\log \rho(B_n)$ and $F_{n}F_{n-1}-F_{n+1}F_{n-2}=(-1)^n$ hold, and the value $\phi^{-2}$ lies strictly between $F_{n-2}/F_{n}$ and $F_{n-1}/F_{n+1}$.
\end{lemma}
\begin{proof}
Define a sequence of finite words by $u_{(1)}:=1$, $u_{(2)}:=0$, and $u_{(n+1)}:=u_{(n)}u_{(n-1)}$ for every $n \geq 2$. Clearly we have $\mathcal{A}(u_{(n)})=B_n$ for all $n \geq 1$.
 A simple induction argument shows that each $u_{(n)}$ is a standard word in the sense defined in Lemma~\ref{standard}, and that $|u_{(n)}|=F_n$, $|u_{(n)}|_1=F_{n-2}$ for every $n \geq 2$.
 By Lemma~\ref{standard} and Lemma~\ref{Sestimates}(i) we therefore have $[u_{(n)}]^\infty \in X_{F_{n-2}/F_n}$ and consequently $S(F_{n-2}/F_{n}) = F_n^{-1}\log \rho(\mathcal{A}(u_{(n)}))=F_n^{-1}\log \rho(B_n)$ for every $n \geq 2$ as required. The remaining parts of the lemma follow from the fact that the fractions $F_{n-2}/F_n$ are precisely the continued fraction convergents of $\phi^{-2}$. Alternatively these results can be derived from the explicit formula for $(F_n)$.
\end{proof}
\begin{lemma}\label{taun-sequence-1}
For each $n \geq 1$ we have $\tr B_n = \tau_n$.
\end{lemma}
\begin{proof}
By direct evaluation the reader may obtain $\tr B_1 = \tr B_2 =2=\tau_1 = \tau_2$ and $\tr B_3 =3= \tau_3$, so it suffices to show that the sequence $(\tr B_n)$ satisfies the same recurrence relation as $(\tau_n)$ for all $n \geq 3$. Let us write
\[B_n = \left(\begin{array}{cc}a_n&b_n\\c_n&d_n\end{array}\right)\]
for each $n\geq 1$. Notice that we have $a_nd_n-b_nc_n = \det B_n = 1$  for every $n$, and for each $n \geq 2$ the definition $B_{n+1}:=B_nB_{n-1}$ implies the identity
\[\left(\begin{array}{cc}a_{n+1}&b_{n+1}\\ c_{n+1}&d_{n+1}\end{array}\right) = \left(\begin{array}{cc}a_na_{n-1}+b_nc_{n-1}& a_nb_{n-1}+b_nd_{n-1}\\c_na_{n-1}+d_nc_{n-1}& c_nb_{n-1}+d_nd_{n-1}\end{array}\right).\]
Fix any $n \geq 3$. By definition we have
\[\tr B_{n+1} = a_{n+1}+d_{n+1} = a_n a_{n-1} + b_nc_{n-1} + c_n b_{n-1} + d_n d_{n-1}\]
and
\[(\tr B_n)(\tr B_{n-1})= a_n a_{n-1}  + a_n d_{n-1} + d_na_{n-1}+d_n d_{n-1},\]
so we may compute
\begin{align*}(\tr B_n)(\tr B_{n-1})-\tr B_{n+1} &= a_n d_{n-1} +d_n a_{n-1}- b_n c_{n-1} - c_n b_{n-1}\\
&= d_{n-1}(a_{n-1}a_{n-2} + b_{n-1}c_{n-2}) + a_{n-1}(c_{n-1}b_{n-2}+d_{n-1}d_{n-2})\\
&\quad- c_{n-1}(a_{n-1}b_{n-2} + b_{n-1}d_{n-2}) - b_{n-1}(c_{n-1}a_{n-2}+d_{n-1}c_{n-2})\\
&= a_{n-2}(a_{n-1}d_{n-1}-b_{n-1}c_{n-1}) + d_{n-2}( a_{n-1}d_{n-1}-b_{n-1}c_{n-1} )\\
& = a_{n-2}+d_{n-2} = \tr B_{n-2},\
\end{align*}
which establishes the required recurrence relation.
\end{proof}

\begin{lemma}\label{taun-sequence-2}
There exist constants $\delta_1,\delta_2>0$ such that
\[\left|\log \tau_n - \log \rho(B_n)\right| =  O\left(e^{-\delta_1F_n}\right)\]
and
\[\left|\log \left(1-\frac{\tau_{n-1}}{\tau_{n+1}\tau_{n}}\right)\right| = O\left(e^{-\delta_2F_n}\right)\]
in the limit as $n \to \infty$.
\end{lemma}
\begin{proof}
It is clear that $F_{n-2}/F_n \to \phi^{-2}$ using the formula for $F_n$, and since $S$ is continuous it follows via Lemma~\ref{bn-matrices} that $F_n^{-1}\log \rho(B_n) \to S(\phi^{-2})>0$. Since $\det B_n=1$ and $B_n$ is non-negative, the eigenvalues of $B_n$ are $\rho(B_n)$ and $\rho(B_n)^{-1}$ respectively, so for each $n \geq 1$ we have $\tau_n = \tr B_n = \rho(B_n) + \rho(B_n)^{-1}$, where we have used Lemma~\ref{taun-sequence-1}. Hence,
\[0 \leq \log \tau_n - \log \rho(B_n) = \log \left(\frac{\rho(B_n)+\rho(B_n)^{-1}}{\rho(B_n)}\right)\leq \frac{1}{\rho(B_n)^2} = O\left(e^{-F_n S(\phi^{-2})}\right),\]
where we have used the elementary inequality $\log (1+x) \leq x$ which holds for all real $x$, and this proves the first part of the lemma.

It follows from this result that $\lim_{n \to \infty} F_n^{-1} \log \tau_n = S(\phi^{-2})$. We may therefore apply this to obtain

\begin{align*}
\lim_{n \to \infty}\frac{1}{F_n}\log \left(\frac{\tau_{n-1}}{\tau_{n+1}\tau_{n}}\right) &= \lim_{n \to \infty}\left(\frac{1}{F_n} \log \tau_{n-1} - \frac{1}{F_n} \log \tau_{n+1}- \frac{1}{F_n} \log \tau_{n}\right)\\&= S(\phi^{-2})(\phi^{-1} - \phi - 1) = -2 S(\phi^{-2})<0,
\end{align*}
from which the second part of the lemma follows easily.
\end{proof}

\noindent \emph{Proof of Theorem~\ref{counter}}.
We will show that $S'(\phi^{-2})=-\log \alpha_*$, where $\alpha_*$ satisfies the product and limit formulas given in the statement of the Theorem. By Lemma~\ref{deriv} this implies that $\mathfrak{r}(\alpha_*)=\phi^{-2} \notin \mathbb{Q}$, and by Theorem \ref{technical} this implies that $\mathsf{A}_{\alpha_*}$ does not satisfy the finiteness property.

By Lemma~\ref{deriv} together with Lemma~\ref{FIM}, the derivative $S'(\phi^{-2})$ exists and is finite. Using Lemma~ \ref{bn-matrices} and Lemma~\ref{taun-sequence-2}, we may now compute
\begin{align*}
S'(\phi^{-2})&=\lim_{n \to \infty}\frac{S\left(\frac{F_{n-1}}
{F_{n+1}}\right)-S\left(\frac{F_{n-2}}
{F_{n}}\right)}{\frac{F_{n-1}}
{F_{n+1}}-\frac{F_{n-2}}{F_{n}}}\\
&=\lim_{n \to \infty}\frac{\frac{1}{F_{n+1}}
\log\rho\left(B_{n+1}\right)-\frac{1}{F_{n}}\log
\rho\left(B_{n}\right)}{\frac{F_{n-1}}
{F_{n+1}}-\frac{F_{n-2}}{F_{n}}}\\
&=\lim_{n \to \infty}\frac{{F_{n}}
\log\rho\left(B_{n+1}\right)-{F_{n+1}}\log
\rho\left(B_{n}\right)}{F_nF_{n-1}-
F_{n+1}F_{n-2}}\\
&=\lim_{n \to
\infty}(-1)^n(F_{n}\log\rho\left(B_{n+1}\right)-F_{n+1}\log
\rho(B_n))\\
&=\lim_{n \to \infty}(-1)^n(F_{n}\log\tau_{n+1}-F_{n+1}\log
\tau_n).\end{align*}
Let us define
\[\alpha_* := e^{-S'(\phi^{-2})} =  \lim_{n \to \infty}
\left(\frac{\tau_n^{F_{n+1}}}{\tau_{n+1}^{F_n}}\right)^{(-1)^n}
\]
which yields the first of the two expressions for $\alpha_*$. We shall derive the second expression. Let us write $\alpha_n:= (\tau_n^{F_{n+1}} / \tau_{n+1}^{F_n})^{(-1)^n}$ for each $n \geq 1$ so that $\alpha_* = \lim_{n \to \infty} \alpha_n$. Applying the recurrence relations for $(F_n)$ and $(\tau_n)$ once more, we obtain for each $n \geq 1$
\begin{align*}\frac{\alpha_{n+1}}{\alpha_n}&= \frac{\left(\tau_{n+1}^{F_{n+2}}/ \tau_{n+2}^{F_{n+1}}\right)^{(-1)^{n+1}}} {\left(\tau_n^{F_{n+1}}/\tau_{n+1}^{F_n}\right)^{(-1)^n}} = \left(\frac{\tau_{n+2}^{F_{n+1}} \tau_{n+1}^{F_n}}{\tau_{n+1}^{F_{n+2}}\tau_n^{F_{n+1}}}\right)^{(-1)^n}\\ &= \left(\frac{\tau_{n+2}}{\tau_{n+1}\tau_{n}}\right)^{(-1)^n F_{n+1}} = \left(\frac{\tau_{n+1}\tau_n - \tau_{n-1}}{\tau_{n+1}\tau_{n}}\right)^{(-1)^n F_{n+1}} = \left(1-\frac{\tau_{n-1}}{\tau_{n+1}\tau_{n}}\right)^{(-1)^nF_{n+1}}.
\end{align*}
Since $\tau_1=\tau_2=2$ and $F_1 = F_2 = 1$ we have $\alpha_1=1$. Using the formula above we may now obtain for each $N\geq 2$
\[
\alpha_{N} = \alpha_1 \prod_{n=1}^{N-1}\frac{\alpha_{n+1}}{\alpha_n} = \prod_{n=1}^{N-1}\left(1-\frac{\tau_{n-1}} {\tau_{n+1}\tau_n}\right)^{(-1)^nF_{n+1}}.
\]
It follows from Lemma~\ref{taun-sequence-2} that these partial products converge unconditionally in the limit $N \to \infty$, and taking this limit we obtain the desired infinite product expression for $\alpha_*$.\qed

\begin{remark} The proof of Theorem~\ref{counter} may be extended to give an explicit estimate for the difference $|\alpha_*-\alpha_N|$ as follows. Note that for each $n \geq 3$ we have $1/3 \leq F_{n-2}/F_n \leq 1/2$ and therefore, by Proposition~\ref{Sproposition},
\begin{align*}
F_n^{-1}\log \tau_n &\geq F_n^{-1}\log \rho(B_n) =  S\left(\frac{F_{n-2}}{F_n}\right) \\ &\geq  S\left(\frac{1}{3}\right)=\frac{\log \rho(A_0^2A_1)}{3}=\frac{\log (2+\sqrt{3})}{3}.
\end{align*}
On the other hand, if we define a sequence $(\tilde\tau_n)_{n=1}^\infty$ by $\tilde\tau_1=\tilde \tau_2 =\tau_1=\tau_2=2$ and $\tilde\tau_{n+1}:=\tilde\tau_n \tilde\tau_{n-1}$ for $n \geq 3$, then it is clear that $\tau_n \leq \tilde\tau_n = 2^{F_n}$ for every $n \geq 1$. Combining these estimates yields
\begin{align*}\left|\log\alpha_N-\log\alpha_*\right| &\leq \sum_{n=N}^\infty F_{n+1} \left|\log \left(1-\frac{\tau_{n-1}}{\tau_{n+1}\tau_n}\right)\right| \leq 2\sum_{n=N}^\infty \frac{F_{n+1}\tau_{n-1}}{\tau_{n+1}\tau_n}\\
&\leq 2\sum_{n=N}^\infty F_{n+1}\frac{2^{F_{n-1}}}{(2+\sqrt{3})^{F_{n+2}/3}}< C_1 \sum_{n=N}^\infty (\phi^{n+1}+1)\theta^{\phi^{n}}\\&\leq 120\sum_{n=N}^\infty \left(\frac{3}{4}\right)^{\phi^n}< 780\left(\frac{3}{4}\right)^{\phi^N}\end{align*}
for all $N \geq 3$, where
\[C_1:=\frac{4(2+\sqrt{3})^{1/3}}{\sqrt{5}}= 2.77475\ldots,\qquad\theta:=\left(\frac{8} {(2+\sqrt{3})^{\phi^3}}\right)^{\frac{1}{3\phi \sqrt{5}}}=0.72441\ldots\]
It follows in particular that the value $\alpha_{13}:=\tau_{14}^{F_{13}}/\tau_{13}^{F_{14}}$  satisfies $|\alpha_*-\alpha_{13}|<10^{-62}$, which yields the approximation given in the introduction.
\end{remark}

\section{Further questions}
\label{sec9}

\noindent {\bf 1.} {\sl Is it true that $\alpha_*$ is irrational or transcendental?} The fast rate of convergence of the sequence $\left(\frac{\tau_n^{F_{n+1}}}{\tau_{n+1}^{F_n}}\right)^{(-1)^n}$ suggests that $\alpha_*$ is probably irrational; however, perhaps unexpectedly, this rate itself is not fast enough to claim this. Roughly, to apply known results (see, e.g., \cite{Nabut}), we need $\tau_n$ to grow like $A^{B^n}$ with $A>1$ and $B>2$. Then Theorem~1 from the aforementioned paper would apply. In our setting however we ``only'' have $B=\phi<2$.

A good illustration how tight the quoted result is is the famous Cantor infinite product
\[
\prod_{n=0}^\infty \left(1+\frac1{2^{2^n}}\right)
\]
equal to 2, despite its ``superfast'' convergence rate. However, a similar product
\[
\prod_{n=0}^\infty \left(1+\frac1{2^{3^n}}\right)
\]
is indeed irrational. We conjecture that $\alpha_{**}=\mathfrak{r}^{-1}(1-1/\sqrt2)$ (which corresponds to the substitution $0\to001,\ 1\to0$ similarly to $\alpha_*$ corresponding to the Fibonacci substitution  $0\to01,\ 1\to0$) is irrational.

\medskip\noindent
{\bf 2.}  {\sl Is $\mathfrak{r}^{-1}(\gamma)$ always a point when $\gamma$ is irrational?} We know this to be true if $\gamma$ is not Liouville (i.e., for all irrational $\gamma$ except a set of zero Hausdorff dimension) but the method used in Lemma~\ref{FIM} is somewhat limited. We hope to close this gap in a follow-up paper.

\medskip\noindent
{\bf 3.} {\sl If the answer to the previous question is yes, then is it true that $\mathfrak{r}^{-1}(\gamma)\notin\mathbb Q$ whenever $\gamma\notin\mathbb Q$?} This question is pertinent to a conjecture of Blondel and Jungers, which says that the finiteness property holds for all matrices with rational entries \cite{BJ}. Our model should not, therefore, yield a counterexample to this conjecture.

\medskip\noindent
{\bf 4.} {\sl Is $\mathfrak{r}^{-1}(\gamma)$ always an interval with nonempty interior when $\gamma$ is rational?} It was shown by the fourth named author in his thesis \cite{Theys} that $\mathfrak{r}^{-1}\left(\frac12\right)=\left[\frac45,1\right]$, and all other known examples indicate that the answer is positive. However proving this for a general $\gamma\in\mathbb Q$ seems like a difficult question.

\medskip\noindent
{\bf 5.} {\sl Does the set of all $\alpha$ such that $\mathfrak{r}(\alpha)\notin \mathbb{Q}$ have zero measure? Does it have zero Hausdorff dimension?}
Analogues of these properties are claimed for Bousch-Mairesse's example but proofs are not given \cite{BM}.

We conjecture that the graph of $\mathfrak{r}$ is a devil's staircase with the plateau regions corresponding to $\{\gamma: \mathfrak r(\gamma)\in\mathbb Q\}$ -- see Figure~\ref{fig:frakr(gamma)}.

\begin{figure}[H]
\[ \includegraphics[width=250pt,height=250pt,angle=270]{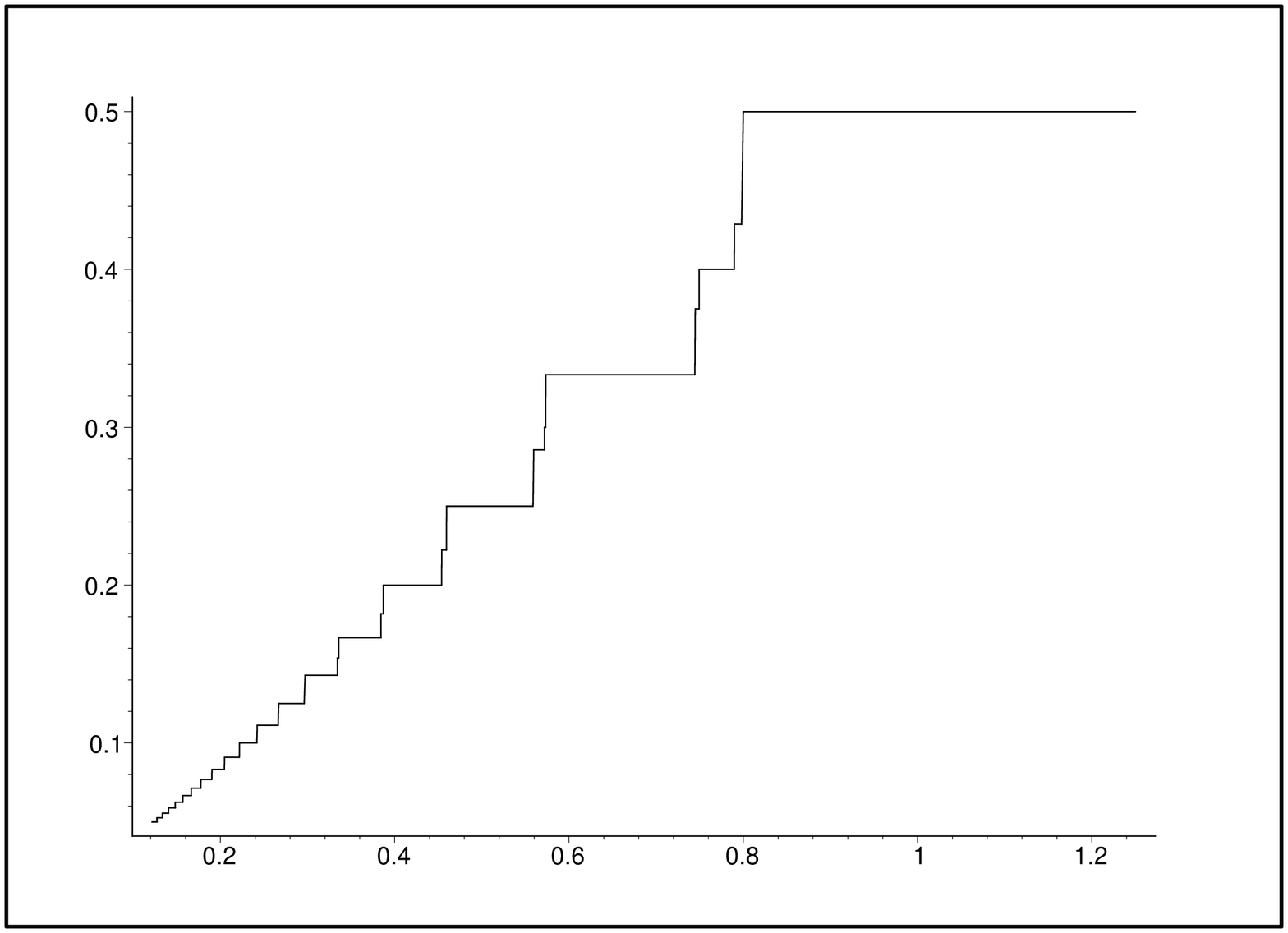} \]
\caption{Graph of $\mathfrak{r}(\gamma)$}
\label{fig:frakr(gamma)}
\end{figure}

\begin{remark}
Between the time of submisssion and present, some progress has been made on
    some of the questions above.
Interested readers are welcome to contact the authors above to find out the
    current progress on these problems.
\end{remark}

\section*{Acknowledgement}
The authors are indebted to V.~S.~Kozyakin for his helpful remarks
    and suggestions.

\def\cprime{$'$} \def\cprime{$'$}

\end{document}